\documentclass[leqno,final]{siamltex}%[10pt]{siamltex}%{amsart}%{amsart}%[]{article}
\usepackage{amsmath, dsfont}
\usepackage{graphicx,color}
\usepackage{amssymb}

\usepackage{bm}
\usepackage{float}
\usepackage{setspace}
\usepackage{subfig}
\usepackage{multirow}
\usepackage[labelfont=bf]{caption}

\newtheorem{remark}{Remark}

\newcommand{\E}{\mathcal{E}}

\newcommand{\Eb}[1]{\E \left[ {#1} \right]}

\usepackage{amsopn}
\def\O{\Omega}
\def\D{\mathcal{D}}
\def\F{\mathcal{F}}
\def\P{\mathbb{P}}
\def\E{\mathbb{E}}

\def\<{\left\langle}
\def\>{\right\rangle}

\begin{document}
	
	\title{Analysis of a mixed finite element method for stochastic Cahn-Hilliard equation with multiplicative noise}
		
	\author{
Yukun Li\thanks{Department of Mathematics,
           University of Central Florida, Orlando, FL 32816 ({\tt yukun.li@ucf.edu}). The work of this author was partially supported by the NSF grant DMS-2110728.}
\and
Corey Prachniak\thanks{Department of Mathematics,
           University of Central Florida, Orlando, FL 32816 ({\tt prachniak@knights.ucf.edu}). The work of this author was partially supported by the NSF grant DMS-2110728.}
\and
Yi Zhang\thanks{Department of Mathematics and Statistics,
          The University of North Carolina at Greensboro, NC 27402  ({\tt y\_zhang7@uncg.edu}). The work of
this author was partially supported by the NSF grant DMS-2111004.}
}

\maketitle

	\begin{abstract} 
	This paper proposes and analyzes a novel fully discrete finite element scheme with the interpolation operator for stochastic Cahn-Hilliard equations with functional-type noise. The nonlinear term satisfies a one-side Lipschitz condition and the diffusion term is globally Lipschitz continuous. The novelties of this paper are threefold. First, the $L^2$-stability ($L^\infty$ in time) and the discrete $H^2$-stability ($L^2$ in time) are proved for the proposed scheme. The idea is to utilize the special structure of the matrix assembled by the nonlinear term. None of these stability results has been proved for the fully implicit scheme in existing literature due to the difficulty arising from the interaction of the nonlinearity and the multiplicative noise. Second, the higher moment stability in $L^2$-norm of the discrete solution is established based on the previous stability results. Third, the H\"older continuity in time for the strong solution is established under the minimum assumption of the strong solution. Based on these, the discrete $H^{-1}$-norm of the strong convergence is discussed. Several numerical experiments including stability and convergence are also presented to validate our theoretical results.
	\end{abstract}
	


		\section{Introduction}
Consider the following stochastic Cahn-Hilliard (SCH) problem:
\begin{align}
\label{eq:SCH}
du = \Bigl[ - \Delta \Big(\epsilon \Delta u - \frac{1}{\epsilon} f(u) \Big) \Bigr] dt + \delta g(u)\,dW_t & \quad \text{in} \ \D_T:= \D \times (0,T],\\
\label{eq:SCH:b}
\frac{\partial u}{\partial n} = \frac{\partial}{\partial n} \Big(\epsilon \Delta u - \frac{1}{\epsilon} f(u)\Big) = 0 & \quad \text{in} \ \partial \D_T:=\partial \D \times (0,T], \\
\label{eq:SCH:i}
u = u_0 & \quad \text{in} \ \D \times \{ 0 \},
\end{align}
where $\D \subset \mathbb{R}^d$ ($d = 2, 3$) is a bounded domain, $n$ is the unit outward normal, $\delta>0$ is a positive constant, and $W_t$ is a standard real-valued Wiener process on a filtered probability space $(\O, \F, \{ \F_t: t \geq 0 \}, \P)$. The function $f$ is the derivative of a smooth double equal well potential function $F$ given by
\begin{align}\label{eq:F}
F(u) = \frac{1}{4} (u^2 - 1)^2. 
\end{align}
The diffusion term $g$ is assumed to have zero mean, is globally Lipschitz continuous (\ref{eq20180812_2}), and satisfies the growth condition (\ref{eq20180813_1}), i.e.,
\begin{align}\label{eq20180812_2}
|g(a)-g(b)|&\le C|a-b|,\\
|g(a)|^2 &\le C(1+a^2).\label{eq20180813_1}
\end{align}
The above formulation of the SCH problem can be rewritten in the following mixed formulation by substituting in the so-called chemical potential $w:= -\epsilon \Delta u + \frac{1}{\epsilon} f(u)$:
\begin{align}\label{Msch1:mIto}
du = \Delta w \, dt + \delta g(u)\,d W_t & \qquad \text{in} \ \D_T, \\
\label{Msch2:mIto}
w = -\epsilon \Delta u + \frac{1}{\epsilon} f(u) & \qquad \text{in} \ \D_T,\\
\label{Msch3:mIto}
\frac{\partial u}{\partial n} = \frac{\partial w}{\partial n} = 0 & \qquad \text{on} \ \partial \D_T, \\
\label{Msch4:mIto}
u = u_0 & \qquad \text{on} \ \D \times \{ 0 \}, 
\end{align}
%where $B = X \otimes X \in  \R^{d \times d}$ with $B_{ij} = X_i X_j$ ($i,j = 1, ..., d$). 
which will be used to develop the fully discrete finite element scheme in this paper.

The deterministic Cahn-Hilliard equation was originally introduced in \cite{Cahn_Hilliard58} to describe phase separation and coarsening processes in a melted alloy. It was proved that the chemical potential approaches the Hele-Shaw problem as the interactive length $\epsilon$ decreases to $0$ \cite{Chen96,ABC:1994,P1989,Stoth96}. Numerical justification for this approximation can be found in \cite{Du_Feng19,feng2016analysis,Feng_Prohl04,Feng_Prohl05,li2017error,wu2018analysis}. We refer to some other references \cite{condette2011spectral,Du_Nicolaides91,Elliott_French89,Eyre98,guan2014second,shen2010numerical,xu2016convex} about numerical approximation for the Cahn-Hilliard equation and the references therein. For stochastic cases, the Cahn-Hilliard-Cook equation with additive noise (with fixed $\epsilon$) was studied in \cite{chai2018conforming,LM2011-CHC,KLM2011-CHC,KLM2014-CHC,furihata2018strong,kossioris2012finite,qi2020error}. The well-posedness of the stochastic Cahn-Hilliard equation was discussed in \cite{da1996stochastic} for additive noise and in \cite{cardon2001cahn} for multiplicative noise. The stability and error estimates in the discrete $H^{-1}$-norm for the stochastic Cahn-Hilliard equation with gradient-type noise were derived in \cite{feng2020fully}. The convergence of the stochastic Cahn-Hilliard equation to the Hele-Shaw flow was considered in \cite{antonopoulou2021numerical,Antonopoulou2018}. The semigroup approach is used for the stochastic Cahn-Hilliard equation with additive noise in \cite{bavnas2023robust,cui2018numerical,cui2021strong,furihata2018strong}, and the references therein.

The goal of the paper is to design a numerical scheme for stochastic Cahn-Hilliard equations with functional-type noise, and hence to explore some numerical properties of the proposed scheme. Here the nonlinear term is not globally Lipschitz continuous and only satisfies a one-side Lipschitz continuity, which interacts with the diffusion term to add another layer of difficulty in designing numerical schemes. We are aiming to establish some stability results, higher moment bounds, and finally the convergence rates for the proposed numerical scheme.

The rest of the paper is organized as follows. In Section \ref{sec2}, the weak solution is defined and several H\"older continuity results are obtained. In Section \ref{sec3}, the scheme is designed, and several stability results and higher moment bounds are both derived for the proposed scheme. In Section \ref{sec4}, the strong convergence with discrete $H^{-1}$-norm is established for the scheme. In Section \ref{sec5}, numerical experiments are done to validate our theoretical results based on different initial conditions and diffusion terms.

\section{Preliminary}\label{sec2}
Throughout this paper, $C$ is used to denote a generic constant, and we use the standard Sobolev
notations in \cite{BS2008}. Additionally, $(\cdot\ ,\ \cdot)$ will denote the standard inner product of $L^2(\D)$, and $\E\left[\cdot\right]$ denotes the expectation operator on the filtered probability space $(\Omega,\F,\{\F_t : t\ge0\}, \P)$.

In this section, we first define the following weak formulation for problem \eqref{eq:SCH}--\eqref{eq:SCH:i}: 
seeking an $\F_{t}$-adapted and $H^1(\D)\times H^1(\D)$-valued process $(u(\cdot,t), w(\cdot,t))$ such that  
there hold $\P$-almost surely 
\begin{align}
(u(t),\phi)_{} &=  (u_0,\phi)_{} - \int_0^t (\nabla w(s), \nabla \phi)_{} \, ds \label{SCH1:mIto:w1}\\
&\quad + \delta \int_0^t (g(u) d W_s, \phi)_{} \qquad \forall \, \phi \in H^1(\D) \quad \forall \, t \in (0,T], \notag\\
(w(t), \varphi) & = \epsilon (\nabla u(t), \nabla \varphi) + \frac{1}{\epsilon} ( f(u(t)), \varphi )  \qquad \forall \, \varphi \in H^1(\D)  \quad \forall \, t \in (0,T]. \label{SCH1:mIto:w2}
\end{align}

Next, we derive several H\"older continuity results in time for $u$ with respect to the spatial $L^2$-norm and for $w$ with respect to the spatial $H^1$-seminorm. These results are used in the error analysis section because the time derivative of $u$ does not exist.

\begin{lemma}\label{lem:e3}
Let $u$ be the strong solution to problem \eqref{Msch1:mIto}--\eqref{Msch4:mIto}.
Then for any $s,t \in [0,T]$ with $t < s$, we have
\begin{align*}
\mathbb{E} \left [ \|  u(s) - u(t) \|_{L^2}^2 \right ] + \epsilon \mathbb{E} \left [ \int_t^s \| \Delta (u(\xi) - u(t)) \|_{L^2}^2 \, d \xi\right ] \leq C_1 (s-t),
\end{align*}
where
\begin{align*}
C_1 = C \sup_{t \leq \xi \leq s} \mathbb{E} \left [ \| \Delta u (\xi) \|^2_{L^2} \right] + C \sup_{t \leq \xi \leq s} \mathbb{E} \left [ \| u (\xi) \|^6_{H^1} \right].
\end{align*}
\end{lemma}

\begin{proof}
Consider the functional $\psi$ defined as:
$$
\psi(u(s)) = \| u(s) - u(t) \|_{L^2}^2.
$$

The first two Gateaux derivatives are:
\begin{align*}
D\psi(u(s)) ( \nu_1 ) %&= \lim_{m \rightarrow 0} \frac{1}{m} [ \| \nabla u(s) + m \nabla \nu_1(s) - \nabla u(t) \|_{L^2}^2 - \| \nabla u(s) - \nabla u(t) \|_{L^2}^2 ]
%\\\\
%&= \lim_{m \rightarrow 0} \frac{1}{m} [ (m\nabla\nu_1(s), \nabla u(s) + m \nabla \nu_1(s) - u(t)) + (\nabla u(s) - \nabla u(t), m \nabla \nu_1(s)) ]
%\\\\
&= 2(  u(s) -  u(t) ,  \nu_1(s) ) ,\\
D\psi(u(s)) (\nu_1, \nu_2 ) %&= \lim_{m \rightarrow 0} \frac{2}{m} [ ( \nabla u(s) + m \nabla \nu_2(s) - \nabla u(t) , \nabla \nu_1(s) ) - ( \nabla u(s) - \nabla u(t) , \nabla \nu_1(s) )]
%\\\\
&= 2 (  \nu_1(s),  \nu_2(s) ).
\end{align*}

Applying It\^o's formula to $\psi(u(s))$ yields
\begin{align*}
\|  u(s) -  u(t) \|_{L^2}^2 &= 2\int_t^s (  u(\xi) - u(t),  -\Delta(\epsilon \Delta u (\xi) - \frac{1}{\epsilon} f (u(\xi))))\, d \xi
\\
&\quad+ \delta^2 \int_t^s ( g(u(\xi)),  g(u(\xi))) \, d \xi
\\
&\quad+2\int_t^s ( u(\xi) - u(t), \delta g(u(\xi)) ) \, dW_{\xi}.
\end{align*}

Now after using integration by parts twice on the first integral, we have
\begin{align*}
\|u(s) - u(t) \|_{L^2}^2 &= 2\int_t^s ( \Delta ( u(\xi) - u(t)), -\epsilon \Delta ( u(\xi) - u(t) ) ) \, d \xi
\\
&\quad+2\int_t^s (  \Delta ( u(\xi) - u(t)), - \epsilon  \Delta u(t) ) \, d \xi
\\
&\quad+2\int_t^s ( \Delta ( u(\xi) - u(t)), \frac{1}{\epsilon}  f(u(\xi)) ) \, d \xi
\\
&\quad+ \delta^2 \int_t^s ( g(u(\xi)),  g(u(\xi))) \, d \xi
\\
&\quad+2\int_t^s ( (u(\xi) - u(t)),  \delta g(u(\xi)) ) \, dW_{\xi}.% := \sum_{i = 1}^5 T_i.
\end{align*}

For the nonlinear term, by using the embedding theorem, we have
\begin{align*}
&\mathbb{E} \left [ \int_t^s \| f(u(\xi)) \|_{L^2}^2 \, d \xi \right ] \\
&\qquad = \mathbb{E} \left [ \int_t^s \int_{\mathcal{D}} u^6(\xi) - 2u^4(\xi) + u^2(\xi) \, dx \, d \xi \right ]
\\
&\qquad \leq C \sup\limits_{t \leq \xi \leq s}  \mathbb{E} \left [ \| u(\xi) \|^6_{H^1} \right ](s-t).
\end{align*}

Then upon taking the expectation on both sides as well as applying the Cauchy-Schwarz inequality, the Young's inequality, and the Gronwall's inequality, we have
\begin{align*}
&\mathbb{E} \left [ \| u(s) - u(t) \|_{L^2}^2 \right ] + \epsilon \mathbb{E} \left [ \int_t^s \| \Delta (u(\xi) - u(t)) \|_{L^2}^2 \, d \xi\right ]\\
\leq& C \sup_{t \leq \xi \leq s} \mathbb{E} \left [ \| \Delta u (\xi) \|^2_{L^2} \right] ( s - t )
+ C \sup_{t \leq \xi \leq s} \mathbb{E} \left [ \| u (\xi) \|^6_{H^1} \right] ( s - t ).
\end{align*}

The conclusion is proved.
\end{proof}

\begin{lemma}\label{20220808_1}
For any $s,t \in [0,T]$ with $t < s$, the chemical potential $w$ satisfies
\begin{align*}
\Eb{\|\nabla w(s)-\nabla w(t)\|_{L^2(\D)}^2} \leq C_2(s-t),
\end{align*}
where
\begin{align*}
C_2 = C\sup_{t\leq\zeta\leq s}\Eb{\| u(\zeta)\|_{H^6}^6}.
\end{align*}
\end{lemma}

\begin{proof}
First define $g(u(s)) = g_1(u(s)) + g_2(u(s))$ where
\begin{align*}
g_1(u(s)) &= \| \epsilon \nabla \Delta u(s) - \epsilon \nabla \Delta u(t) \|^2_{L^2},
\\
g_2(u(s)) &= \| \frac{1}{\epsilon} \nabla f(u(s)) - \frac{1}{\epsilon} \nabla f(u(t))\|^2_{L^2}.
\end{align*}
The first two Gateaux derivatives of $g_1$ are:
\begin{align*}
D g_1(u(s))(\nu_1) &= 2 \epsilon^2 \int_{\mathcal{D}} ( \nabla \Delta u(s) - \nabla \Delta u(t) ) \cdot \nabla \Delta \nu_1(s) \, dx,
\\
D g_1(u(s))(\nu_1, \nu_2) &= 2 \epsilon^2 \int_{\mathcal{D}} \nabla \Delta \nu_1 \cdot \nabla \Delta \nu_2 \, dx.
\end{align*}
The first two Gateaux derivatives of $g_2$ are:
\begin{align*}
D g_2(u(s))(\nu_1) &= \frac{2}{\epsilon^2} \int_{\mathcal{D}} [3u^2(s) \nabla u(s) - \nabla u(s) - \nabla f(u(t))]
\\
&\cdot [6u(s) \nu_1(s) \nabla u(s) + 3u^2(s) \nabla \nu_1(s) - \nabla \nu_1(s)] \, dx
\\
D g_2(u(s))(\nu_1, \nu_2) &= \frac{2}{\epsilon^2} \int_{\mathcal{D}} [3u^2(s) \nabla u(s) - \nabla u(s) - \nabla f(u(t))]
\\
&\cdot [6 \nu_2(s) \nu_1(s) \nabla u(s) + 6u(s) \nu_1(s) \nabla \nu_2(s) + 6u(s) \nu_2(s) \nabla \nu_1(s)] \, dx
\\
&+ \frac{2}{\epsilon^2} \int_{\mathcal{D}} [6 u(s) \nu_1(s) \nabla u(s) + 3u^2(s) \nabla \nu_1(s) - \nabla \nu_1(s)]
\\
&\cdot [3u^2(s) \nabla \nu_2(s) + 6 u(s) \nu_2(s) \nabla u(s) - \nabla \nu_2(s)] \, dx.
\end{align*}
Then applying the It\^{o}'s formula to $g(w(s)) = \| \nabla w(s) - \nabla w(t) \|^2_{L^2}$ yields
\begin{align*}
&\| \nabla w(s) - \nabla w(t) \|^2_{L^2} = 2 \epsilon^2 \int_t^s ( \nabla \Delta (u(\xi) - u(t)) , \nabla \Delta G_1(\xi) ) \, d \xi
\\
&\quad+ \epsilon^2 \int_t^s (\nabla \Delta G_2(\xi), \nabla \Delta G_2(\xi)) \, d \xi + 
2 \epsilon^2 \int_t^s ( \nabla \Delta (u(\xi) - u(t)) , \nabla \Delta G_2(\xi) ) \, dW_\xi
\\
&\quad+ \frac{2}{\epsilon^2}\int_t^s \int_{\mathcal{D}} [3u^2(\xi) \nabla u(\xi) - \nabla u(\xi) - \nabla f(u(t))]
\\
&\qquad \qquad \quad \cdot [6u(\xi) G_1(\xi) \nabla u(\xi) + 3u^2(\xi) \nabla G_1(\xi) - \nabla G_1(\xi)] \, dx \, d \xi
\\
&\quad+ \frac{2}{\epsilon^2} \int_t^s \int_{\mathcal{D}} [3u^2(\xi) \nabla u(\xi) - \nabla u(\xi) - \nabla f(u(t))]
\\
&\qquad \qquad \quad \cdot [6u(\xi) G_2(\xi) \nabla u(\xi) + 3u^2(\xi) \nabla G_2(\xi) - \nabla G_2(\xi)] \, dx \, d W_\xi
\\
&\quad+ \frac{\delta^2}{\epsilon^2} \int_t^s \int_{\mathcal{D}} [3u^2(\xi) \nabla u(\xi) - \nabla u(\xi) - \nabla f(u(t))]
\\
&\qquad \qquad \quad \cdot [6 G_2^2(\xi) \nabla u(\xi) + 6u(\xi) G_2(\xi) \nabla G_2(\xi) + 6u(\xi) G_2(\xi) \nabla G_2(\xi)] \, dx \, d \xi
\\
&\quad+ \frac{\delta^2}{\epsilon^2} \int_t^s \int_{\mathcal{D}} [6 u(\xi) G_2(\xi) \nabla u(\xi) + 3u^2(\xi) \nabla G_2(\xi) - \nabla G_2(\xi)]
\\
&\qquad \qquad \quad \cdot [3u^2(\xi) \nabla G_2(\xi) + 6 u(\xi) G_2(\xi) \nabla u(\xi) - \nabla G_2(\xi)] \, dx \, d \xi,
%\\&:= \sum_{n = 1}^7 T_n
\end{align*}
where
\begin{align*}
G_1(\xi) &= -\Delta \left ( \epsilon \Delta u(\xi) - \frac{1}{\epsilon} f(u(\xi)) \right ),
\\
G_2(\xi) &= \delta g(u(\xi)).
\end{align*}

Taking the expectation on both sides of the above equation, and using the Young's inequality and the embedding theorem, we get
\begin{align}\label{eq20230530_1}
\|\nabla w(s) -\nabla w(t)\|_{L^2}^2
\le C\sup_{t\leq\zeta\leq s}\Eb{\| u(\zeta)\|_{H^6}^6}(s-t),
\end{align}
where the term $H^6$ norm is from the term $\frac{2}{\epsilon^2}\int_t^s \int_{\mathcal{D}} 9u^4(\xi) \nabla u(\xi)\nabla G_1(\xi)dxd\xi$.

\end{proof}

\section{Stability and Convergence}\label{sec3}
Let $t_n = n \tau$ ($n = 0, 1, ..., N$) be a uniform partition of $[0, T]$ and $\mathcal{T}_h$ be a quasi-uniform triangulation of $\D$. Define $V_h$ to be the finite element space given by 
\begin{equation}
V_h := \{ v_h \in H^1(\D): v_h\mid_K \in \mathcal{P}_1(K) \quad \forall \, K \in \mathcal{T}_h \},  
\end{equation}
where $\mathcal{P}_1(K)$ denotes the space of linear polynomials on the element $K$. Define $\mathring{V}_h$ be the subspace of  $V_h$ with zero mean, i.e., 
\begin{align} 
\mathring{V}_h := \bigl\{ v_h \in V_h:\, (v_h, 1)= 0 \bigr\}. 
\end{align}

The fully discrete mixed finite element methods for \eqref{Msch1:mIto}-\eqref{Msch4:mIto} is to seek $\F_{t_n}$-adapted and $V_h \times V_h$-valued process $\{ (u_h^n, w_h^n) \}$ ($n = 1, \dots, N$) such 
that $\P$-almost surely
\begin{align}\label{eq20211119_1}
(u^n_h - u^{n-1}_h, \eta_h) + \tau (\nabla w^n_h, \nabla \eta_h) &= \delta (g(u^{n-1}_h), \eta_h) \Delta W^n \quad &\forall \eta_h \in V_h,
\\
\epsilon (\nabla u^n_h, \nabla v_h) + \frac{1}{\epsilon} (I_h f^n, v_h) &= (w^n_h, v_h) \quad &\forall v_h \in V_h,\label{eq20211119_2}
\end{align}
where $\Delta W^{n}:=W^{n}-W^{n-1}$ satisfies the normal distribution $\mathcal{N}(0,1)$, $f^{n} := (u^{n}_h)^3-u^{n}_h$, and $I_h : C(\bar{\D}) \rightarrow V_h$ is the standard nodal value interpolation operator defined by
\[
I_hv := \sum_{i = 1}^{N_h} v(a_i) \psi_i.
\]
Here $N_h$ denotes the number of vertices of $\mathcal{T}_h$ and $\psi_i$ denotes the nodal basis for $V_h$ corresponding to the vertex $a_i$. The initial conditions are chosen to be $u^0_h = P_h u_0$ and $w^0_h = P_h w_0$, where $P_h : L^2( \D ) \rightarrow V_h$ is the $L^2$-projection operator defined by
\[
(P_hv, v_h) = (v, v_h) \qquad \forall v_h \in V_h.
\]

The following properties of the $L^2$-projection can be found in \cite{BS2008,ciarlet2002finite}:
\begin{align}
\label{Ph1}
&\|v - P_h v \|_{L^2} + h \| \nabla (v - P_h v) \|_{L^2} 
\leq C h^{\min\{2,s\}} \|v\|_{H^s},\\ 
\label{Ph2}
&\|v - P_h v\|_{L^\infty} \leq C h^{2-\frac{d}{2}} \|v\|_{H^2}. 
\end{align}
for all $v \in H^s(\D)$ such that $s>\frac32$.
Furthermore, the inverse discrete 
Laplace operator $\Delta_h^{-1}: \mathring{V}_h \rightarrow \mathring{V}_h$ is defined by
\begin{equation}\label{eq:distInvLap}
\bigl( \nabla(-\Delta_h^{-1}\zeta_h),\nabla v_h \bigr) = \bigl( \zeta_h, v_h \bigr) \qquad \forall\, v_h\in \mathring{V}_h.
\end{equation} 
Then for $\zeta_h, \Phi_h \in \mathring{V}_h$ the discrete $H^{-1}$ inner product is defined as
\begin{equation}\label{eq3.3}
(\zeta_h, \Phi_h)_{-1, h} = \bigl(\nabla(-\Delta_h^{-1}\zeta_h),\nabla(-\Delta_h^{-1}\Phi_h) \bigr)
=\bigl(\zeta_h,-\Delta_h^{-1}\Phi_h\bigr)=\bigl(-\Delta_h^{-1}\zeta_h,\Phi_h\bigr). 
\end{equation}
It is easy to show that the discrete $H^{-1}$-norm satisfies the following two properties
\begin{alignat}{2}
|(\zeta_h, \Phi_h)| &\leq \| \zeta_h \|_{-1, h} |\Phi_h|_{H^1}&&\qquad\forall\zeta_h, \Phi_h \in \mathring{V}_h,
\\
\|\zeta_h \|_{-1, h} &\leq C \| \zeta_h \|_{L^2}&&\qquad\forall\zeta_h\in \mathring{V}_h.
\end{alignat}
% {\bf Scheme 2}. %The fully discrete mixed finite element methods for \eqref{Msch1:mIto}-\eqref{Msch4:mIto} is defined as seeking
% Seeking $\F_{t_n}$-adapted  and $V_h \times V_h$-valued process $\{ (u_h^n, w_h^{n-1}) \}$ ($n = 1, \dots, N$) such 
% that $\P$-almost surely
% \begin{align}\label{eq20230507_1}
% (u^n_h - u^{n-1}_h, \eta_h) + \tau (\nabla w^n_h, \nabla \eta_h) &= \delta (g(u^{n-1}_h), \eta_h) \Delta W^n \quad &\forall \eta_h \in V_h,
% \\
% \epsilon (\nabla u^n_h, \nabla v_h) + \frac{1}{\epsilon} (f^n, v_h) &= (w^n_h, v_h) \quad &\forall v_h \in V_h.\label{eq20230507_2}
% \end{align}
Lastly, define the discrete Laplace operator $\Delta_h : V_h \rightarrow V_h$ by
\begin{equation}\label{eq20230613_2}
(\Delta_h \zeta_h, v_h) = -(\nabla \zeta_h, \nabla v_h) \qquad \forall v_h \in V_h.
\end{equation}

%We will assume the following mesh constraints:
%\begin{align}\label{mesh_constraint}
%\tau \leq C \epsilon^3, 
%\end{align}
%for a positive constant $C$.

\begin{theorem}\label{thm:stab:2}
%Suppose the mesh constraint $\tau \leq C \epsilon^3$ holds, and 
Let $u^n_h\ (n=1,2,\cdots N)$ be the solution of \eqref{eq20211119_1} and \eqref{eq20211119_2}, then there exists a positive constant $C$ independent of $h$ and $\tau$ that 
\begin{align}\label{stab:LinfyL2}
\max_{1 \leq n \leq N} \E\left[\|u^n_h\|^2_{L^2}\right] + \sum_{n=1}^N\E\left[\|u^n_h - u^{n-1}_h\|^2_{L^2}\right] + \tau\sum_{n=1}^N\E\left[\|\Delta_h u^n_h\|^2_{L^2}\right]\le C.
\end{align}
%where $C$ depends on $\delta$ and $\epsilon$.
\end{theorem}

\begin{proof}
Testing \eqref{eq20211119_1} and \eqref{eq20211119_2} by $\eta_h = u^n_h$ and $v_h = -\Delta_h u^n_h$, respectively, we get
\begin{align}
(u^n_h - u^{n-1}_h, u^n_h) + \tau (\nabla w^n_h, \nabla u^n_h) &= \delta (g(u^{n-1}_h), u^n_h) \Delta W^n,\label{eq20211012_1}\\
\epsilon (\nabla u^n_h, - \nabla \Delta_h u^n_h) + \frac{1}{\epsilon} (I_h f^n, -\Delta_h u^n_h) &= (w^n_h, -\Delta_h u^n_h).\label{eq20211012_2}
\end{align}
Thus,
\begin{align}\label{eq20211012_3}
\frac{1}{2}\|u^n_h\|^2_{L^2} - \frac{1}{2}\|u^{n-1}_h\|^2_{L^2} + \frac{1}{2}\|u^n_h - u^{n-1}_h\|^2_{L^2} + \tau& (\nabla w^n_h, \nabla u^n_h) \\
&= \delta (g(u^{n-1}_h), u^n_h) \Delta W^n,\notag
\\
\epsilon \|\Delta_h u^n_h\|^2_{L^2} + \frac{1}{\epsilon} (\nabla I_h f^n, \nabla u^n_h) &= (\nabla w^n_h, \nabla u^n_h).\label{eq20211012_4}
\end{align}

Multiplying the second equation by $\tau$ and substituting into the first equation, and then rearranging, we get 
\begin{align}\label{eq20211119_8}
&\frac{1}{2}\|u^n_h\|^2_{L^2} - \frac{1}{2}\|u^{n-1}_h\|^2_{L^2} + \frac{1}{2}\|u^n_h - u^{n-1}_h\|^2_{L^2} + \epsilon \tau \|\Delta_h u^n_h\|^2_{L^2}\\
&\qquad=\delta (g(u^{n-1}_h), u^n_h) \Delta W^n - \frac{\tau}{\epsilon}(\nabla I_h f^n, \nabla u^n_h).\notag
\end{align}

One key role of the interpolation operator is to bound the nonlinear term. Denote $u_{i}=u_{h}^n\left(a_{i}\right)$, and then
\begin{align}
- \frac{\tau}{\epsilon}(\nabla I_h f^n, \nabla u^n_h)=& \frac{\tau}{\epsilon}\|\nabla u^n_h\|^2_{L^2} - \frac{\tau}{\epsilon} ( \nabla \sum_{i = 1}^{N_h} u_i^3 \varphi_i, \nabla \sum_{i = j}^{N_h} u_j \varphi_j)\notag
\\
=& \frac{\tau}{\epsilon} \|\nabla u^n_h\|^2_{L^2} - \frac{\tau}{\epsilon} \sum_{i, j = 1}^{N_h} b_{ij} (\nabla \varphi_i, \nabla \varphi_j),\notag
\end{align}
where $b_{ij}=u_i^3 u_j$. Notice when $i \neq j$, we have
\begin{equation}
 b_{i j} \leq \frac{3}{4} u_{i}^{4}+\frac{1}{4} u_{j}^{4}.
\end{equation}
The stiffness matrix is diagonally dominant, and then
\begin{align}
&-\frac{\tau}{\epsilon} \sum_{i, j=1}^{N_{h}} b_{i j}\left(\nabla \varphi_{i}, \nabla \varphi_{j}\right) \\
&\qquad\leq-\frac{\tau}{\epsilon} \sum_{k=1}^{N_{h}} b_{k k}\bigg[\left(\nabla \varphi_{k}, \nabla \varphi_{k}\right)-\frac{3}{4} \sum_{i=1 \atop i \neq k}^{N_{h}}\mid\left(\nabla \varphi_{i}, \nabla \varphi_{k}\right)\mid-\frac{1}{4} \sum_{j=1 \atop j \neq k}^{N_{h}}\mid\left(\nabla \varphi_{k}, \nabla \varphi_{j}\right)\mid\bigg]\notag \\
&\qquad\leq-\frac{\tau}{\epsilon} \sum_{k=1}^{N_{h}} b_{k k}\bigg[\left(\nabla \varphi_{k}, \nabla \varphi_{k}\right)-\sum_{i=1 \atop i \neq k}^{N_{h}}\mid\left(\nabla \varphi_{i}, \nabla \varphi_{k}\right)\mid\bigg]\notag\\
&\qquad\leq 0.\notag
\end{align}
Based on this inequality, we have
\begin{align}\label{eq20230401_9}
- \frac{\tau}{\epsilon}(\nabla I_h f^n, \nabla u^n_h) \leq& \frac{\tau}{\epsilon}\left\|\nabla u_{h}^n\right\|_{L^{2}}^{2},\\
\leq& \frac{C \tau}{\epsilon^3} \|u^n_h\|^2_{L^2} + \frac{\epsilon \tau}{2} \|\Delta_h u^n_h\|^2_{L^2}.\notag
\end{align}

For the diffusion term, apply the growth property of $g(\cdot)$ and the martingale property to obtain 
\begin{align}\label{eq20211119_5}
&\delta \mathbb{E}\left[(g(u^{n-1}_h), u^n_h) \Delta W^n\right] \\
&\qquad= \delta \mathbb{E}\left[(g(u^{n-1}_h), u^n_h - u^{n - 1}_h) \Delta W^n\right]\notag\\
&\qquad\leq \frac{1}{4} \mathbb{E}\left[\|u^n_h - u^{n-1}_h\|^2_{L^2}\right] +  \delta^2\tau \mathbb{E}\left[\|g(u^{n-1}_h)\|^2_{L^2}\right]\notag\\
&\qquad\leq \frac{1}{4} \mathbb{E} \left[\|u^n_h - u^{n-1}_h\|^2_{L^2}\right] + C\delta^2\tau + C\delta^2 \tau \E\left[\|u^{n-1}_h\|^2_{L^2}\right].\notag
\end{align}

Taking the expectation and taking the summation over $n$ from $1$ to $\ell$ in \eqref{eq20211119_8}, we have
\begin{align}\label{eq20211119_6}
&\E\left[\frac12\|u^\ell_h\|^2_{L^2}\right] + \frac14\sum_{n=1}^\ell\E\left[\|u^n_h - u^{n-1}_h\|^2_{L^2}\right] + \frac{\epsilon\tau}{2}\sum_{n=1}^\ell\E\left[\|\Delta_h u^n_h\|^2_{L^2}\right]\\
&\qquad\le \frac{C \tau}{\epsilon^3}\sum_{n=1}^\ell\E\left[\|u^n_h\|^2_{L^2}\right] + C \delta^2 + C\delta^2 \tau \sum_{n=1}^\ell\E\left[\|u^{n-1}_h\|^2_{L^2}\right] + \E\left[\frac12\|u^0_h\|^2_{L^2}\right]\notag.
\end{align}
%By the mesh constraint $\tau \leq C \epsilon^3$ and t
By the discrete Gronwall inequality, we obtain
\begin{align}\label{eq20211119_9}
\E\left[\frac12\|u^\ell_h\|^2_{L^2}\right] + \frac14\sum_{n=1}^\ell\E\left[\|u^n_h - u^{n-1}_h\|^2_{L^2}\right] + \frac{\epsilon\tau}{2}\sum_{n=1}^\ell\E\left[\|\Delta_h u^n_h\|^2_{L^2}\right]\le C,
\end{align}
where $C$ depends on $\delta$ and $\epsilon$. 
Finally, the estimate \eqref{stab:LinfyL2} follows from \eqref{eq20211119_9}. 
\end{proof}

By equation \eqref{eq20230613_2}, the Cauchy-Schwartz inequality, and Theorem \ref{thm:stab:2}, we could directly obtain the following Corollary.
\begin{corollary}\label{coro1}
Let $u^n_h\ (n=1,2,\cdots N)$ be the solution of \eqref{eq20211119_1} and \eqref{eq20211119_2}, then there holds 
\begin{align}\label{eq20230613_1}
\tau\sum_{n=1}^N\E\left[\|\nabla u^n_h\|^2_{L^2}\right]\le C.
\end{align}
\end{corollary}

\begin{theorem}\label{thm:stab:p}
Suppose the mesh constraint $\tau \leq C \epsilon^3$ holds, and let $u^n_h\ (n=1,2,\cdots N)$ be the solution of \eqref{eq20211119_1} and \eqref{eq20211119_2}, then there exists a positive constant independent of $h$ and $\tau$ such that there holds for any $p \geq 2$ that
\begin{align}\label{stab:LinftyL2：p}
\sup _{0 \leq n \leq N} \mathbb{E}\left[\left\|u_{h}^{n}\right\|_{L^{2}}^{p}\right] \leq C.
\end{align}
\end{theorem}

\begin{proof}
The proof is divided into three steps. In Step 1, we prove the bound for $\mathbb{E}\left\|u_{h}^{n}\right\|_{L^{2}}^{4}$. In Step 2, we establish the bound for $\mathbb{E}\left\|u_{h}^{n}\right\|_{L^{2}}^{p}$, where $p=2^{r}$ and $r$ is an arbitrary positive integer. In Step 3, using Step 1 and Step 2, we could obtain the bound for $\mathbb{E}\left\|u_{h}^{n}\right\|_{L^{2}}^{p}$, where $p$ is an arbitrary real number and $p \geq 2$.

{\bf Step 1.} Based on $\eqref{eq20211119_8}-\eqref{eq20230401_9}$, we have
\begin{align}\label{eq20220108_1}
&\frac{1}{2}\|u^n_h\|^2_{L^2} - \frac{1}{2}\|u^{n-1}_h\|^2_{L^2} + \frac{1}{2}\|u^n_h - u^{n-1}_h\|^2_{L^2} + \epsilon \tau \|\Delta_h u^n_h\|^2_{L^2}
\\
&\qquad \leq \frac{\tau}{\epsilon}\|\nabla u_{h}^n\|_{L^{2}}^{2} + \delta (g(u^{n-1}_h), u^n_h) \Delta W^n. \notag
\end{align}

Multiplying the quantity
\begin{align}\label{eq20220108_2}
\|u_{h}^{n}\|_{L^{2}}^{2}+\frac{1}{2}\| u_{h}^{n-1}\|_{L^{2}}^{2} &= \frac{3}{4}\left(\|u_{h}^{n}\|_{L^{2}}^{2}+\| u_{h}^{n-1}\|_{L^{2}}^{2}\right)+\frac{1}{4}\left(\| u_{h}^{n}\|_{L^{2}}^{2}-\| u_{h}^{n-1}\|_{L^{2}}^{2}\right).
\end{align}

on both sides of \eqref{eq20220108_1} gives us
\begin{align}\label{eq20211119_10}
&\quad\frac{3}{8}(\| u_{h}^{n}\|_{L^{2}}^{4}-\| u_{h}^{n-1}\|_{L^{2}}^{4})+\frac{1}{8}(\| u_{h}^{n}\|_{L^{2}}^{2}-\| u_{h}^{n-1}\|_{L^{2}}^{2})^{2}
\\
&\qquad +(\frac{1}{2}\|u^n_h - u^{n-1}_h\|^2_{L^2} + \epsilon \tau \|\Delta_h u^n_h\|^2_{L^2})(\| u_{h}^{n}\|_{L^{2}}^{2}+\frac{1}{2}\|u_{h}^{n-1}\|_{L^{2}}^{2}). \notag
\\
&\leq \notag
\frac{\tau}{\epsilon}\|\nabla u_h^n\|^2_{L^2} (\| u_{h}^{n}\|_{L^{2}}^{2}+\frac{1}{2}\|u_{h}^{n-1}\|_{L^{2}}^{2})
\\
&\qquad + \delta (g(u^{n-1}_h), u^n_h) \Delta W^n(\|u_{h}^{n}\|_{L^{2}}^{2}+\frac{1}{2}\| u_{h}^{n-1}\|_{L^{2}}^{2}). \notag
\end{align}

The first term on the right-hand side of \eqref{eq20211119_10} can be estimated as
\begin{align}\label{eq20211119_11}
&\quad \frac{\tau}{\epsilon}\|\nabla u_h^n\|^2_{L^2} (\| u_{h}^{n}\|_{L^{2}}^{2}+\frac{1}{2}\|u_{h}^{n-1}\|_{L^{2}}^{2}) \\
&\le\frac{\epsilon}{2}\tau \|\Delta_h u_h^n\|^2_{L^2} (\| u_{h}^{n}\|_{L^{2}}^{2}+\frac{1}{2}\|u_{h}^{n-1}\|_{L^{2}}^{2})\notag\\
&\qquad+\frac{C\tau}{\epsilon^3}\|u_h^n\|^2_{L^2} (\frac{3}{2} \|u_{h}^{n}\|_{L^{2}}^{2} - \frac{1}{2}(\|u_{h}^{n}\|_{L^{2}}^{2} - \|u_{h}^{n-1}\|_{L^{2}}^{2})\notag\\ 
&\leq \frac{\epsilon}{2}\tau \|\Delta_h u_h^n\|^2_{L^2} (\| u_{h}^{n}\|_{L^{2}}^{2}+\frac{1}{2}\|u_{h}^{n-1}\|_{L^{2}}^{2})\notag\\
&\qquad +\frac{C \tau}{\epsilon^3} \|u_h^n\|^4_{L^2} + \frac{C \tau}{\epsilon^3} ( \|u_{h}^{n}\|_{L^{2}}^{2} - \|u_{h}^{n-1}\|_{L^{2}}^{2} )^2. \notag
\end{align}

The second term can be estimated by the Cauchy-Schwarz inequality and the growth property of $g(\cdot)$ as
\begin{align}\label{eq20211119_12}
&\quad \delta (g(u^{n-1}_h), u^n_h) \Delta W^n(\|u_{h}^{n}\|_{L^{2}}^{2}+\frac{1}{2}\| u_{h}^{n-1}\|_{L^{2}}^{2}) \\
&= \delta (g(u^{n-1}_h), u^n_h - u^{n-1}_h + u^{n-1}_h) \Delta W^n(\|u_{h}^{n}\|_{L^{2}}^{2}+\frac{1}{2}\| u_{h}^{n-1}\|_{L^{2}}^{2}) \notag \\
&\leq \left( \frac{1}{4} \|u^{n}_h - u^{n-1}_h\|^2_{L^2} + C\delta^2(1 + \|u^{n-1}_h\|^2_{L^2}) (\Delta W^{n})^2 \right) (\|u_{h}^{n}\|_{L^{2}}^{2}+\frac{1}{2}\| u_{h}^{n-1}\|_{L^{2}}^{2}) \notag \\
&\qquad + \delta (g(u^{n-1}_h), u^{n-1}_h) \Delta W^n 
(\|u_{h}^{n}\|_{L^{2}}^{2}+\frac{1}{2}\| u_{h}^{n-1}\|_{L^{2}}^{2}). \notag
\end{align}

By the Cauchy-Schwarz inequality, we have
\begin{align}\label{eq20211119_13}
&\quad C\delta^2 \left(1+\|u_{h}^{n-1}\|_{L^{2}}^{2}\right)\left({\Delta} W^{n}\right)^{2}\left(\|u_{h}^{n}\|_{L^{2}}^{2}+\frac{1}{2}\|u_{h}^{n-1}\|_{L^{2}}^{2}\right) \\
&=C\delta^2 \left(1+\|u_{h}^{n-1}\|_{L^{2}}^{2}\right)\left({\Delta} W^{n}\right)^{2}\left(\|u_{h}^{n}\|_{L^{2}}^{2}-\|u_{h}^{n-1}\|_{L^{2}}^{2}+\frac{3}{2}\|u_{h}^{n-1}\|_{L^{2}}^{2}\right) \notag \\
&\leq \theta_{1}\left(\|u_{h}^{n}\|_{L^{2}}^{2}-\|u_{h}^{n-1}\|_{L^{2}}^{2}\right)^{2} + C \delta^4 \left(1 + \|u_{h}^{n-1}\|_{L^{2}}^{4}\right)\left({\Delta} W^{n}\right)^{4} \notag \\
&\quad+C\delta^2 \|u_{h}^{n-1}\|_{L^{2}}^{4}\left({\Delta} W^{n}\right)^{2}+C\delta^2 \|u_{h}^{n-1}\|_{L^{2}}^{2}\left({\Delta} W^{n}\right)^{2}, \notag 
\end{align}
where $\theta_1 > 0$ will be specified later. Furthermore, we have 
\begin{align}\label{eq20211119_13_2}
&\quad \delta (g(u_{h}^{n-1}), u_{h}^{n-1}) {\Delta} W^{n}\left(\|u_{h}^{n}\|_{L^{2}}^{2}+\frac{1}{2}\|u_{h}^{n-1}\|_{L^{2}}^{2}\right) \\
&=\delta (g(u_{h}^{n-1}), u_{h}^{n-1}) {\Delta} W^{n}\left(\|u_{h}^{n}\|_{L^{2}}^{2}-\|u_{h}^{n-1}\|_{L^{2}}^{2}+\frac{3}{2}\|u_{h}^{n-1}\|_{L^{2}}^{2}\right) \notag \\
&\leq \theta_{2}\left(\|u_{h}^{n}\|_{L^{2}}^{2}-\|u_{h}^{n-1}\|_{L^{2}}^{2}\right)^{2} + C\delta^2 \left(1+\|u_{h}^{n-1}\|_{L^{2}}^{4}\right)\left(\Delta W^{n}\right)^{2} \notag 
\\
&\quad + \frac{3}{2} \delta (g(u_{h}^{n-1}), u_{h}^{n-1})\|u_{h}^{n-1}\|_{L^{2}}^{2} \Delta W^{n}, \notag 
\end{align}
where $\tau$, $\theta_1$, and $\theta_2$ are chosen to be small enough such that 
\begin{align}
\frac{C \tau}{\epsilon^3} + \theta_{1}+\theta_{2} \leq \frac{1}{16},
\end{align}

Then by the martingale property and properties of the Wiener process, after taking the summation from $n = 1$ to $n = \ell$ and the expectation on both sides of \eqref{eq20211119_10}, we get
\begin{align}
&\left( \frac{3}{8} - \frac{C\tau}{\epsilon^3} \right) \mathbb{E}\left[\|u_{h}^{\ell}\|_{L^{2}}^{4}\right]+\frac{1}{16} \sum_{n=1}^{\ell} \mathbb{E}\left[\left(\| u_{h}^{n}\|_{L^{2}}^{2}-\| u_{h}^{n-1}\|_{L^{2}}^{2}\right)^{2}\right]
\\
&\quad+\sum_{n=1}^{\ell} \mathbb{E}\left[\left(\frac{1}{4}\|u_{h}^{n}-u_{h}^{n-1}\|_{L^{2}}^{2}+\frac{\epsilon}{2}\tau\|\Delta_{h} u_{h}^{n}\|_{L^{2}}^{2}\right)\left(\| u_{h}^{n}\|_{L^{2}}^{2}+\frac{1}{2}\| u_{h}^{n-1}\|_{L^{2}}^{2}\right)\right] \notag 
\\
&\leq \frac{C \tau}{\epsilon^3} \sum_{n=0}^{\ell-1} \mathbb{E}\left[\| u_{h}^{n}\|_{L^{2}}^{4}\right]+\frac{3}{8} \mathbb{E}\left[\| u_{h}^{0}\|_{L^{2}}^{4}\right]+C (\delta^2 + \delta^4) \tau^{2} \sum_{n=0}^{\ell-1} \mathbb{E}\left[\| u_{h}^{n}\|_{L^{2}}^{4}\right] \notag 
\\
&\qquad +C (\delta^2 + \delta^4) \tau \sum_{n=0}^{\ell-1} \mathbb{E}\left[\| u_{h}^{n}\|_{L^{2}}^{4}\right]+C (\delta^2 + \delta^4 \tau). \notag 
\end{align}

Under the mesh constraint $\tau \leq C \epsilon^3$, we have 
\begin{align}
& \frac{1}{4} \mathbb{E}\left[\| u_{h}^{\ell}\|_{L^{2}}^{4}\right]+\frac{1}{16} \sum_{n=1}^{\ell} \mathbb{E}\left[\left(\| u_{h}^{n}\|_{L^{2}}^{2}-\| u_{h}^{n-1}\|_{L^{2}}^{2}\right)^{2}\right]
\\
&\quad+\sum_{n=1}^{\ell} \mathbb{E}\left[\left(\frac{1}{4}\|u_{h}^{n}-u_{h}^{n-1}\|_{L^{2}}^{2}+\frac{\epsilon}{2}\tau\|\Delta_{h} u_{h}^{n}\|_{L^{2}}^{2}\right)\left(\| u_{h}^{n}\|_{L^{2}}^{2}+\frac{1}{2}\| u_{h}^{n-1}\|_{L^{2}}^{2}\right)\right] \notag \\
&\leq C \left( \frac{1}{\epsilon^3} + \delta^2 + \delta^4 \right) \tau \sum_{n=0}^{\ell-1} \mathbb{E}\left[\| u_{h}^{n}\|_{L^{2}}^{4}\right]+\frac{3}{8} \mathbb{E}\left[\| u_{h}^{0}\|_{L^{2}}^{4}\right]+C (\delta^2 + \delta^4). \notag 
\end{align}

By the Gronwall's inequality, we have
\begin{align}
&\frac{1}{4} \mathbb{E}\left[\| u_{h}^{\ell}\|_{L^{2}}^{4}\right]+\frac{1}{16} \sum_{n=1}^{\ell} \mathbb{E}\left[\left(\| u_{h}^{n}\|_{L^{2}}^{2}-\| u_{h}^{n-1}\|_{L^{2}}^{2}\right)^{2}\right] 
\\
&\quad+\sum_{n=1}^{\ell} \mathbb{E}\left[\left(\frac{1}{4}\|u_{h}^{n}-u_{h}^{n-1}\|_{L^{2}}^{2}+\frac{\epsilon}{2}\tau\|\Delta_{h} u_{h}^{n}\|_{L^{2}}^{2}\right)\left(\| u_{h}^{n}\|_{L^{2}}^{2}+\frac{1}{2}\| u_{h}^{n-1}\|_{L^{2}}^{2}\right)\right] \notag\\
&\leq C.\notag
\end{align}

{\bf Step 2.} From \eqref{eq20220108_1}--\eqref{eq20211119_13_2}, we have 
\begin{align}\label{eq20220209_1}
&\quad \frac{1}{4}\left(\| u_{h}^{n}\|_{L^{2}}^{4}-\| u_{h}^{n-1}\|_{L^{2}}^{4}\right)+\frac{1}{16}\left(\| u_{h}^{n}\|_{L^{2}}^{2}-\| u_{h}^{n-1}\|_{L^{2}}^{2}\right)^{2} \\
&\qquad+\left(\frac{1}{4}\|\left(u_{h}^{n}-u_{h}^{n}\right)\|_{L^{2}}^{2}+\frac{\epsilon}{2}\tau\|\Delta_{h} u_{h}^{n}\|_{L^{2}}^{2}\right)\left(\| u_{h}^{n}\|_{L^{2}}^{2}+\frac{1}{2}\| u_{h}^{n-1}\|_{L^{2}}^{2}\right) \notag \\
& \leq C \frac{\tau}{\epsilon^3} \| u_{h}^{n}\|_{L^{2}}^{4}+C \delta^4 (1 + \| u_{h}^{n-1}\|_{L^{2}}^{4}) \left({\Delta} W^{n}\right)^{4} \notag \\
&\qquad +C \delta^2 (1 + \| u_{h}^{n-1}\|_{L^{2}}^{4}) \left({\Delta} W_{n}\right)^{2}+C\delta \| u_{h}^{n-1}\|_{L^{2}}^{4} \Delta W^{n}. \notag 
\end{align}

Similar to Step 1, by multiplying \eqref{eq20220209_1} by $\| u_{h}^{n}\|_{L^{2}}^{4}+\frac{1}{2}\| u_{h}^{n-1}\|_{L^{2}}^{4}$, we could get the 8-th moment of the $L^2$-stability of the numerical solution. Then repeat this process, we could obtain the $2^{r}$-th moment of the $L^2$-stability of the numerical solution.

{\bf Step 3.} Suppose $2^{r-1} \leq p \leq 2^{r}$, and then by Young's inequality, we have
\begin{align}\label{eq20220209_2}
\mathbb{E}\left[\| u_{h}^{\ell}\|_{L^{2}}^{p}\right] \leq \mathbb{E}\left[\| u_{h}^{\ell}\|_{L^{2}}^{2^{r}}\right]+C<\infty, 
\end{align}
where the second inequality follows from the results in Step 2. The proof is complete.
\end{proof}

% \begin{remark}
% The constant $C$ in Theorem~\ref{thm:stab:2} and Theorem~\ref{thm:stab:p} depends on $\epsilon^{-3}$ and $\delta$ exponentially. 
% \end{remark}

\section{Error Estimates}\label{sec4}
Define a sequence of subsets as below
% \begin{align} \label{eq:subset-omega}
% \widetilde{\Omega}_{\kappa,m} = \Bigl\{\omega\in\Omega:
%   \max\limits_{1\leq n \leq m}\|u_h^{n}\|_{L^\infty}^2\le C\bigl(\ln(h^{-\beta})\bigr)^{2}\tau^{-1}\Bigr\}.
% \end{align}
\begin{align} \label{eq:subset-omega}
\widetilde{\Omega}_{\kappa,n} = \Bigl\{\omega\in\Omega:
  \max\limits_{1\leq i \leq n}\|u_h^{i}\|_{H^1}^2\le\kappa\Bigr\},
\end{align}
where $\kappa$ will be specified later. %satisfies $0\le gamma\le 1$ by Corollary \ref{coro1}. \kappa=C|\ln h|\tau^{-gamma} 
Clearly, it holds that $\widetilde{\Omega}_{\kappa, 0} \supset \widetilde{\Omega}_{\kappa,1} \supset \cdots \supset \widetilde{\Omega}_{\kappa,\ell}$. 

Next, for each $n = 0, 1, ..., N$, define
\begin{alignat*}{2}
E^n &:= u(t_n) - u^n_h, \quad &&G^n := w(t_n) - w^n_h,\\
\Theta^n &:= u(t_n) - P_hu(t_n),\quad &&\Lambda^n := w(t_n) - P_hw(t_n),\\
\Phi^n &:= P_hu(t_n) - u^n_h,\quad &&\Psi^n := P_hw(t_n) - w^n_h.
\end{alignat*}

\begin{theorem}
The following error estimate holds for any $\ell=1, 2, \cdots, N$:
\begin{align}\label{20230623_2}
&\mathbb{E}\left[\mathds{1}_{\widetilde{\Omega}_{\kappa, \ell}}\|E^\ell\|^2_{-1,h}\right] + \tau\sum_{n=1}^{\ell}\mathbb{E}\left[\mathds{1}_{\widetilde{\Omega}_{\kappa, n}}\|\nabla E^n\|^2_{L^2(\mathcal{D})}\right]\le C\tau+Ch^2+Ch^4\bigl(\ln h\bigr)^6\tau^{-2},
\end{align}
where $C$ depends on $\sup\limits_{t\in[0,T]}\Eb{\| u(t)\|_{H^6}^6}$ and $\sup\limits_{t\in[0,T]}\Eb{\| w(t)\|_{H^2}^2}$.
\end{theorem}

\begin{proof}
From the weak formulation (\ref{SCH1:mIto:w1})-(\ref{SCH1:mIto:w2}), we get
\begin{align}
(u(t_n), \eta_h) &= (u(t_{n-1}), \eta_h) - \int_{t_{n-1}}^{t_n} (\nabla w(s), \nabla \eta_h) ds \label{eq20220224_1}\\
&\quad + \delta \int_{t_{n-1}}^{t_n} (g(u(s)), \eta_h) dW_s\qquad\quad\ \ \qquad \forall \eta_h\in V_h,\notag\\
(w(t_n), v_h) &= \epsilon (\nabla u(t_n), \nabla v_h) + \frac{1}{\epsilon}(f(u(t_n)), v_h)\qquad\forall v_h\in V_h. \label{eq20220224_2}
\end{align}
hold $\mathbb{P}$-almost surely. Subtracting \eqref{eq20211119_1}-\eqref{eq20211119_2} from \eqref{eq20220224_1}-\eqref{eq20220224_2}, we get
\begin{align}
(E^n, \eta_h) &= (E^{n-1}, \eta_h)- \int_{t_{n-1}}^{t_n} (\nabla w(s) - \nabla w^n_h, \nabla \eta_h) ds\\
&\quad +\delta \int_{t_{n-1}}^{t_n} (g(u(s)) - g(u^{n-1}_h), \eta_h) dW_s\qquad\quad\forall \eta_h\in V_h,\notag\\
(G^n, v_h) &= \epsilon (\nabla E^n, \nabla v_h) + \frac{1}{\epsilon}(f(u(t_n)) - I_hf^n, v_h)\quad\quad\ \forall v_h\in V_h.
\end{align}

Choosing $\eta_h = -\Delta_h^{-1} \Phi^n\in \mathring{V}_h$ and $v_h = \tau \Phi^n\in \mathring{V}_h$ yields
\begin{align}
(E^n - E^{n-1}, -\Delta_h^{-1} \Phi^n) &= - \int_{t_{n-1}}^{t_n} (\nabla (w(s) - w^n_h), \nabla (-\Delta_h^{-1} \Phi^n)) ds \notag
\\
&\quad +\delta \int_{t_{n-1}}^{t_n} (g(u(s)) - g(u^n_h), -\Delta_h^{-1} \Phi^n) dW_s, \notag
\\
(G^n, \tau \Phi^n) &= \epsilon (\nabla E^n, \nabla (\tau \Phi^n)) + \frac{1}{\epsilon}(f(u(t_n)) - I_h f^n, \tau \Phi^n).\notag
\end{align}

Thus we have
\begin{align}\label{eq20230609_1}
(\nabla\Delta_h^{-1} (\Phi^n - \Phi^{n-1}), \nabla\Delta_h^{-1} \Phi^n) &= (\Theta^n - \Theta^{n-1}, \Delta_h^{-1} \Phi^n) \\
&\quad + \tau(\nabla\Lambda^n, \nabla\Delta_h^{-1}\Phi^n) + \tau(\nabla\Psi^n, \nabla\Delta_h^{-1}\Phi^n)\notag\\
&\quad + \int_{t_{n-1}}^{t_n} (\nabla w(s) - \nabla w(t_n), \nabla \Delta_h^{-1} \Phi^n) ds \notag\\
&\quad + \delta \int_{t_{n-1}}^{t_n} (g(u(s)) - g(u^n_h), -\Delta_h^{-1} \Phi^n) dW_s, \notag\\
\tau (\Lambda^n + \Psi^n, \Phi^n) &= \epsilon \tau (\nabla \Theta^n, \nabla \Phi^n) + \epsilon \tau (\nabla \Phi^n, \nabla \Phi^n) \\
&\quad + \frac{\tau}{\epsilon}(f(u(t_n)) - I_h f^n, \Phi^n). \notag
\end{align}

Taking expectation on both sides over the set $\widetilde{\Omega}_{\kappa,n}$ and substituting the second equation into the first yield
\begin{align}\label{eq20230610_1}
&\E\left[\mathds{1}_{\widetilde{\Omega}_{\kappa, n}}(\Phi^n - \Phi^{n-1}, \Phi^n)_{-1, h}\right] + \epsilon \tau \E\left[\mathds{1}_{\widetilde{\Omega}_{\kappa, n}}(\nabla \Phi^n, \nabla \Phi^n)\right] \\
=& \E\left[\mathds{1}_{\widetilde{\Omega}_{\kappa, n}}(\Theta^n - \Theta^{n-1}, \Delta_h^{-1} \Phi^n)\right]
- \epsilon \tau \E\left[\mathds{1}_{\widetilde{\Omega}_{\kappa, n}}(\nabla \Theta^n, \nabla \Phi^n)\right] \notag
\\
&- \frac{\tau}{\epsilon} \E\left[\mathds{1}_{\widetilde{\Omega}_{\kappa, n}}(f(u(t_n)) - I_hf^n, \Phi^n)\right]+ \E\left[\mathds{1}_{\widetilde{\Omega}_{\kappa, n}}\int_{t_{n-1}}^{t_n} (\nabla w(s) - \nabla w(t_n), \nabla \Delta_h^{-1} \Phi^n) ds\right] \notag
\\
&+ \delta \E\left[\mathds{1}_{\widetilde{\Omega}_{\kappa, n}}\int_{t_{n-1}}^{t_n} (g(u(s)) - g(u^n_h), -\Delta_h^{-1} \Phi^ndW_s)\right]+\tau\E\left[\mathds{1}_{\widetilde{\Omega}_{\kappa, n}}(\nabla\Lambda^n, \nabla\Delta_h^{-1}\Phi^n)\right]\notag
\\
:=& \sum_{i = 1}^6 T_i. \notag
\end{align}

The first term on the left-hand side of \eqref{eq20230610_1} can be bounded by
\begin{align}
&\E\left[\mathds{1}_{\widetilde{\Omega}_{\kappa, n}}(\Phi^n - \Phi^{n-1}, \Phi^n)_{-1, h}\right]\\
=& \frac{1}{2}\E\left[\mathds{1}_{\widetilde{\Omega}_{\kappa, n}}\|\Phi^n\|^2_{-1,h}\right] - \frac{1}{2}\E\left[\mathds{1}_{\widetilde{\Omega}_{\kappa, n}}\|\Phi^{n-1}\|^2_{-1,h}\right] + \frac{1}{2}\E\left[\mathds{1}_{\widetilde{\Omega}_{\kappa, n}}\|\Phi^n - \Phi^{n-1}\|^2_{-1,h}\right]\notag\\
=& \frac{1}{2}\E\left[\mathds{1}_{\widetilde{\Omega}_{\kappa, n}}\|\Phi^n\|^2_{-1,h}\right] - \frac{1}{2}\E\left[\mathds{1}_{\widetilde{\Omega}_{\kappa, n-1}}\|\Phi^{n-1}\|^2_{-1,h}\right] + \frac{1}{2}\E\left[\mathds{1}_{\widetilde{\Omega}_{\kappa, n}}\|\Phi^n - \Phi^{n-1}\|^2_{-1,h}\right]\notag\\
&+\frac{1}{2}\E\left[(\mathds{1}_{\widetilde{\Omega}_{\kappa, n-1}}-\mathds{1}_{\widetilde{\Omega}_{\kappa, n}})\|\Phi^{n-1}\|^2_{-1,h}\right]\notag\\
\ge& \frac{1}{2}\E\left[\mathds{1}_{\widetilde{\Omega}_{\kappa, n}}\|\Phi^n\|^2_{-1,h}\right] - \frac{1}{2}\E\left[\mathds{1}_{\widetilde{\Omega}_{\kappa, n-1}}\|\Phi^{n-1}\|^2_{-1,h}\right] + \frac{1}{2}\E\left[\mathds{1}_{\widetilde{\Omega}_{\kappa, n}}\|\Phi^n - \Phi^{n-1}\|^2_{-1,h}\right]\notag.
\end{align}

The first term $T_1$ on the right-hand side of \eqref{eq20230610_1} is zero by the definition of the $P_h$ operator. 

Using the Young's inequality and properties of the projection, the second term $T_2$ can be bounded by
\begin{align}
T_2 &\leq \epsilon \tau\mathbb{E}\left[\mathds{1}_{\widetilde{\Omega}_{\kappa, n}}\|\nabla\Theta^n\|^2_{L^2}\right] + \frac{\epsilon \tau}{4} \mathbb{E}\left[\mathds{1}_{\widetilde{\Omega}_{\kappa, n}}\|\nabla\Phi^n\|^2_{L^2}\right]
\\
&\leq C\tau h^2 \mathbb{E}\left[\mathds{1}_{\widetilde{\Omega}_{\kappa, n}}\mid u(t_n)\mid^2_{H^2}\right] + \frac{\epsilon \tau}{4} \mathbb{E}\left[\mathds{1}_{\widetilde{\Omega}_{\kappa, n}}\|\nabla\Phi^n\|^2_{L^2}\right]. \notag
\end{align}

In order to estimate the third term $T_3$, we write
\begin{align}\label{derrest:7}
( f(u(t_n)) - I_hf^{n}, \Phi^{n} ) &= ( f(u(t_{n}) - f(P_h u(t_{n})), \Phi^{n} ) \\
   &\quad +  (f(P_h u(t_{n})) - f^{n}, \Phi^{n}) + (f^{n}-I_hf^{n}, \Phi^{n}) \notag.
\end{align}
Using the properties of the projection and the embedding theorem, we have
\begin{align}
&-\frac{\tau}{\epsilon}\E\left[\mathds{1}_{\widetilde{\Omega}_{\kappa, n}}\bigl( f(u(t_{n})) - f(P_h u(t_{n})), \Phi^{n} \bigr) \right] \label{derrest:9}\\
&\quad  = -\frac{\tau}{\epsilon} \E \left[\mathds{1}_{\widetilde{\Omega}_{\kappa, n}} \bigl( \Theta^{n} \bigl(\sum_{i=0}^{2}(u(t_{n}))^{i}(P_h u(t_{n}))^{2-i}-1\bigr),\Phi^{n} \bigr) \right] \notag\\
&\quad  \leq C\tau \E \left[\mathds{1}_{\widetilde{\Omega}_{\kappa, n}} \|\sum_{i=0}^{2}(u(t_{n}))^{i}(P_h u(t_{n}))^{2-i}-1\|_{L^{\infty}}^2 \|\Theta^{n}\|_{L^2}^2 \right] + \tau\E \left[\mathds{1}_{\widetilde{\Omega}_{\kappa, n}}\|\Phi^{n}\|_{L^2}^2 \right] \notag\\
&\quad \leq C\tau\Bigl( \E \left[\mathds{1}_{\widetilde{\Omega}_{\kappa, n}}\left(\|P_h u(t_{n})\|_{L^\infty}^{{6}}
+\|u(t_{n})\|_{L^\infty}^{{6}}+|D|^{3}\right) \right] \Bigr)^{\frac{2}{3}} \Bigl( \E \left[\mathds{1}_{\widetilde{\Omega}_{\kappa, n}}\|\Theta^{n}\|_{L^2}^{6} \right] \Bigr)^{\frac13} \notag \\
&\qquad + \tau\E \left[\mathds{1}_{\widetilde{\Omega}_{\kappa, n}}\|\Phi^{n}\|_{L^2}^2 \right] \notag\\
&\quad \leq C\tau h^2+C\tau h^2\E \left[\mathds{1}_{\widetilde{\Omega}_{\kappa, n}}\|u(t_{n})\|_{H^2}^{{6}}\right] + \tau\E \left[\mathds{1}_{\widetilde{\Omega}_{\kappa, n}}\|\Phi^{n}\|_{L^2}^2 \right] \notag.
\end{align}
The second term on the right-hand
side of \eqref{derrest:7} can be bounded by
\begin{align}
\label{derrest:9_1}
-\frac{\tau}{\epsilon}\E \left[\mathds{1}_{\widetilde{\Omega}_{\kappa, n}}\bigl( f(P_h u(t_{n})) - f^{n}, \Phi^{n} \bigr) \right] \leq \frac{\tau}{\epsilon}\E \left[\mathds{1}_{\widetilde{\Omega}_{\kappa, n}}\|\Phi^{n} \|^2_{L^2} \right].
\end{align}
Using the inverse inequality and equation \eqref{eq:subset-omega}, we have
\begin{align}\label{eq20230611_1}
\mathds{1}_{\widetilde{\Omega}_{\kappa, n}}\max\limits_{1\le n\le\ell}\|u_h^n\|_{L^\infty}^2&\le C(\ln h)^2\mathds{1}_{\widetilde{\Omega}_{\kappa, n}}\bigl(\max\limits_{1\le n\le\ell}\|u_h^n\|_{L^2}^2+\max\limits_{1\le n\le\ell}\|\nabla u_h^n\|_{L^2}^2\bigr)\\
%&\le C|\ln h|\bigl(\max\limits_{1\le n\le\ell}\E\left[\|u_h^n\|_{L^2}^2\right]+\max\limits_{1\le n\le\ell}\E\left[\|\Delta_h u_h^n\|_{L^2}^2\right]\bigr)\notag\\
&\le C(\ln h)^2\kappa.\notag
\end{align}
%Here $gamma$ satisfies $0\le gamma\le1$ based on Theorem \ref{thm:stab:2}.
Then by the properties of the interpolation operator, the inverse inequality, and equation \eqref{eq20230611_1}, the third term on the right-hand side of \eqref{derrest:7} can be handled by
\begin{align}\label{eq20180711_17}
&-\frac{\tau}{\epsilon}\E \left[\mathds{1}_{\widetilde{\Omega}_{\kappa, n}}\bigl( f^{n}-I_hf^{n}, \Phi^{n} \bigr) \right]\\
&\qquad\le C\tau\E\left[\mathds{1}_{\widetilde{\Omega}_{\kappa, n}}\|f^{n}-I_hf^{n}\|_{H^{-1}}^2\right]+\frac{\epsilon}{8}\tau\E \left[\mathds{1}_{\widetilde{\Omega}_{\kappa, n}}\|\nabla\Phi^{n}\|_{L^2}^2 \right]\notag\\
&\qquad\le C\tau h^4\E \left[\mathds{1}_{\widetilde{\Omega}_{\kappa, n}}\sum_{K}|(u_h^{n})^3|^2_{H^1(K)} \right]+\frac{\epsilon}{8}\tau\E \left[\mathds{1}_{\widetilde{\Omega}_{\kappa, n}}\|\nabla\Phi^{n}\|^2_{L^2} \right]\notag\\
&\qquad\le C\tau h^4\E \left[\mathds{1}_{\widetilde{\Omega}_{\kappa, n}}\sum_{K}\|(u_h^{n})^2\nabla u_h^{n}\|_{L^2(K)}^2 \right]+\frac{\epsilon}{8}\tau\E \left[\mathds{1}_{\widetilde{\Omega}_{\kappa, n}}\|\nabla\Phi^{n}\|^2_{L^2} \right]\notag\\
&\qquad\le C\tau h^4\E \left[\mathds{1}_{\widetilde{\Omega}_{\kappa, n}}\|\nabla u_h^{n}\|_{L^2}^2\max\limits_{1\leq n \leq m}\|u_h^{n}\|_{L^\infty}^4\right]+\frac{\epsilon}{8}\tau\E \left[\mathds{1}_{\widetilde{\Omega}_{\kappa, n}}\|\nabla\Phi^{n}\|^2_{L^2} \right]\notag\\
&\qquad\le Ch^4\bigl(\ln h\bigr)^4\kappa^2\tau\E \left[\mathds{1}_{\widetilde{\Omega}_{\kappa, n}}\|\nabla u_h^{n}\|_{L^2}^2\right]+\frac{\epsilon}{8}\tau\E \left[\mathds{1}_{\widetilde{\Omega}_{\kappa, n}}\|\nabla\Phi^{n}\|^2_{L^2} \right]\notag.
\end{align}
Combine \eqref{derrest:7}--\eqref{eq20180711_17} to obtain
\begin{align} \label{derrest:11}
T_3 \leq& C\tau\E\left[\mathds{1}_{\widetilde{\Omega}_{\kappa, n}}\|\Phi^{n}\|_{L^2}^2 \right]+\frac{\epsilon}{8}\tau\E \left[\mathds{1}_{\widetilde{\Omega}_{\kappa, n}}\|\nabla\Phi^{n}\|^2_{L^2} \right]+Ch^2\tau\\
&+Ch^2\tau\E \left[\mathds{1}_{\widetilde{\Omega}_{\kappa, n}}\|u(t_{n})\|_{H^2}^{{6}}\right]+Ch^{6-d}\bigl(\ln(h^{-\beta})\bigr)^8\tau^{-4}\tau\E \left[\mathds{1}_{\widetilde{\Omega}_{\kappa, n}}\|\nabla u_h^{n}\|_{L^2}^2\right]\notag\\
\leq& C\tau\E\left[\mathds{1}_{\widetilde{\Omega}_{\kappa, n}}\|\Phi^{n}\|_{-1,h}^2 \right]+\frac{\epsilon}{4}\tau\E \left[\mathds{1}_{\widetilde{\Omega}_{\kappa, n}}\|\nabla\Phi^{n}\|^2_{L^2} \right]+Ch^2\tau\notag\\
&+Ch^2\tau\sup_{t\in[0,T]}\E \left[\mathds{1}_{\widetilde{\Omega}_{\kappa, n}}\|u(t)\|_{H^2}^{{6}}\right]+Ch^4\bigl(\ln h\bigr)^4\kappa^2\tau\E \left[\mathds{1}_{\widetilde{\Omega}_{\kappa, n}}\|\nabla u_h^{n}\|_{L^2}^2\right]\notag.
% \leq& C\tau\E\left[\mathds{1}_{\widetilde{\Omega}_{\kappa, n}}\|\Phi^{n}\|_{-1,h}^2 \right]+\frac{\epsilon}{4}\tau\E \left[\mathds{1}_{\widetilde{\Omega}_{\kappa, n}}\|\nabla\Phi^{n}\|^2_{L^2} \right]+Ch^2\tau\notag\\
% &+Ch^2\tau\sup_{t\in[0,T]}\E \left[\mathds{1}_{\widetilde{\Omega}_{\kappa, n}}\|u(t)\|_{H^2}^{{6}}\right]+Ch^4\bigl(\ln h\bigr)^6\tau^{-2}\notag,
\end{align}
Notice that the last term can be bounded by Corollary \ref{coro1} after taking the summation.

By Lemma \ref{20220808_1}, the fourth term $T_4$ on the right-hand side of \eqref{eq20230610_1} can be bounded by
\begin{align}
T_4 \leq&\E\left[\mathds{1}_{\widetilde{\Omega}_{\kappa, n}}\int_{t_{n-1}}^{t_n}2 \|\nabla w(s) - \nabla w(t_n)\|^2_{L^2} + C\|\Phi^n\|^2_{-1,h}ds \right]\\
\le&C\tau^2\sup_{t\in[0,T]}\Eb{\| u(t)\|_{H^6}^6}+\tau\mathbb{E}\left[\mathds{1}_{\widetilde{\Omega}_{\kappa, n}}\|\Phi^n\|^2_{-1,h}\right].\notag
\end{align}

By Lemma \ref{lem:e3}, the fifth term $T_5$ on the right-hand side of \eqref{eq20230610_1} can be estimated as:
\begin{align}\label{eq20230610_2}
T_5 &= \delta \E\left[\mathds{1}_{\widetilde{\Omega}_{\kappa, n}}\int_{t_{n-1}}^{t_n} (g(u(s)) - g(u^n_h), -\Delta_h^{-1}\Phi^ndW_s)\right]\\
&\le C\tau^2\bigl(\sup_{t\in[0,T]} \mathbb{E} \left [ \| \Delta u (t) \|^2_{L^2} \right] + C \sup_{t\in[0,T]} \mathbb{E} \left [ \| u (t) \|^6_{H^1} \right]\bigr)+C\tau\E\left[\mathds{1}_{\widetilde{\Omega}_{\kappa, n}}\|\Theta^n\|_{L^2}^2\right]\notag\\
&\quad+C\tau\E\left[\mathds{1}_{\widetilde{\Omega}_{\kappa, n}}\|\Phi^n\|_{L^2}^2\right]+\frac14\mathbb{E}\left[\mathds{1}_{\widetilde{\Omega}_{\kappa, n}}\|\Delta_h^{-1}\Phi^n - \Delta_h^{-1}\Phi^{n-1}\|^2_{L^2}\right]\notag\\
&\le C\tau^2\bigl(\sup_{t\in[0,T]} \mathbb{E} \left [ \|  \Delta u (t) \|^2_{L^2} \right]+ C \sup_{t\in[0,T]} \mathbb{E} \left [ \| u (t) \|^6_{H^1} \right]\bigr)+C\tau h^4\sup_{t\in[0,T]}\Eb{\| u(t)\|_{H^2}^2} \notag\\
&\quad+ \frac{\epsilon\tau}{4} \mathbb{E}\left[\mathds{1}_{\widetilde{\Omega}_{\kappa, n}}\|\nabla\Phi^n\|^2_{L^2}\right]+C\tau \mathbb{E}\left[\mathds{1}_{\widetilde{\Omega}_{\kappa, n}}\|\Phi^n\|^2_{-1,h}\right]+\frac14\mathbb{E}\left[\mathds{1}_{\widetilde{\Omega}_{\kappa, n}}\|\Phi^n - \Phi^{n-1}\|^2_{-1,h}\right].\notag
\end{align}

The sixth term $T_6$ on the right-hand side of \eqref{eq20230610_1} can be bounded by
\begin{align}
T_6 &\leq C\tau h^2\sup_{t\in[0,T]}\Eb{\| w(t)\|_{H^2}^2}+\tau\mathbb{E}\left[\|\Phi^n\|^2_{-1,h}\right].
\end{align}

By Theorem \ref{thm:stab:2} and Corollary \ref{coro1}, we have
\begin{equation}\label{eq20230622_1}
\max_{1 \leq n \leq N} \E\left[\|\nabla u^n_h\|^2_{H^1}\right]\le C\tau^{-1}.
\end{equation}
Choosing $\kappa:=C|\ln h|\tau^{-1}$, using equation \eqref{eq20230622_1}, the Markov's inequality, and the discrete Burkholder--Davis--Gundy inequalities \cite{beiglbock2015pathwise,burkholder1966martingale,burkholder1970extrapolation,davis1970intergrability}, we have
\begin{align}\label{eq20230611_2}
\P\left[\widetilde{\Omega}_{\kappa,\ell}\right]&\ge 1-\frac{\E\left[\max\limits_{1\le n\le\ell}\|u_h^n\|_{H^1}^2\right]}{C|\ln h|\tau^{-1}}\\
&\ge 1- \frac{C}{|\ln h|}\rightarrow1\ \text{as}\ h\rightarrow0\notag.
\end{align}

Combining \eqref{eq20230610_1}--\eqref{eq20230611_2} and using the discrete Gronwall inequality and Corollary \eqref{coro1}, we obtain the conclusion.
\end{proof}

\begin{remark}
If the following stability result holds:
\begin{equation*}
\max_{1 \leq n \leq N} \E\left[\|\nabla u^n_h\|^2_{H^1}\right]\le C,
\end{equation*}
then the last term on the right-hand side of \eqref{20230623_2} will be $Ch^4\bigl(\ln h\bigr)^6$ by choosing $\kappa:=C|\ln h|$. Furthermore, if we have
\begin{equation*}
\max_{1 \leq n \leq N} \E\left[\|\nabla u^n_h\|^p_{H^1}\right]\le C,
\end{equation*}
we can remove the probability one set $\widetilde{\Omega}_{\kappa,n}$ in the error estimates, i.e.,
\begin{align}
&\mathbb{E}\left[\|E^\ell\|^2_{-1,h}\right] + \tau\sum_{n=1}^{\ell}\mathbb{E}\left[\|\nabla E^n\|^2_{L^2(\mathcal{D})}\right]\le C\tau+Ch^2.\notag
\end{align}
\end{remark}

\section{Numerical Experiments}\label{sec5}
This section presents the results from three numerical tests. In all tests, the domain chosen is $\mathcal{D} = [-1, 1] \times [-1, 1]$. The purpose of Tests 1 and 2 is to check the stability and evolution of the numerical approximations from (\ref{eq20211119_1}) under two noise intensities, $\delta = 1, 10$. Test 1 is based on the initial condition with a circular zero-level set. Test 2 is based on the initial condition with an ellipse as its zero-level set. In Test 3, we compute the $L^{\infty} \mathbb{E} L^2$ and $\mathbb{E}L^2 H^1$ errors (\ref{eqn_error}) to check for the theoretical orders of convergence.

\textbf{Test 1:}
For this test, we use the initial condition $u^0_h = P_hu_0$ where
$$
u_0(x, y) = \tanh \left ( \frac{x^2 + y^2 - 0.6^2}{\sqrt{2} \epsilon} \right ),
$$
and the diffusion term is $g(u) = u$. The figures (\ref{fig:my_label1}) and (\ref{fig:my_label2}) show the $\mathbb{E} L^2$ and $\mathbb{E} H^1$ stability results of the numerical solution in one sample and the average of samples, respectively.
Notice that the stability curves are bounded under both diffusion intensities. Also, the $\mathbb{E} H^1$ stability is larger as the diffusion intensity is increased.
\captionsetup[subfigure]{labelformat=empty}
\begin{figure}[h!]%
    \renewcommand{\thefigure}{5.\arabic{figure}}
    \pagecolor{white}
    \centering
    \subfloat[\centering $\delta = 1$]{{\includegraphics[width=6cm]{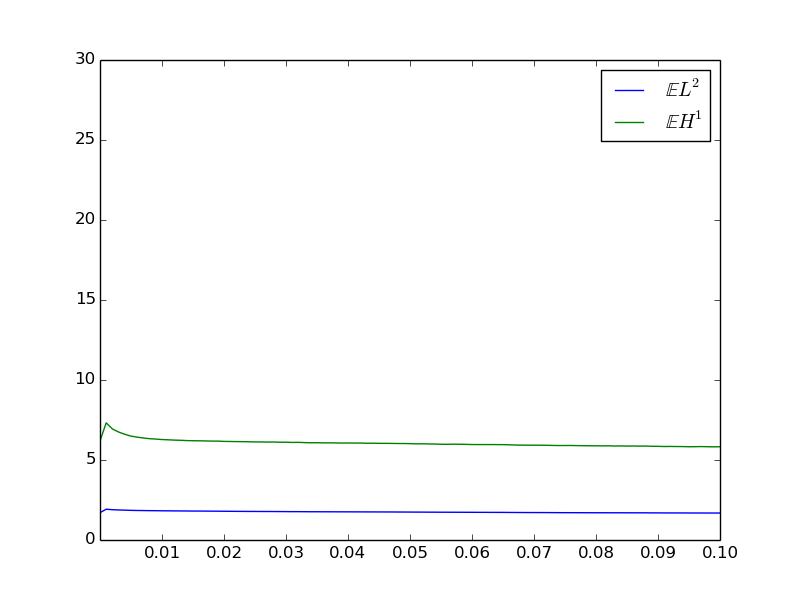}}}%
    \qquad
    \subfloat[\centering $\delta = 10$]{{\includegraphics[width=6cm]{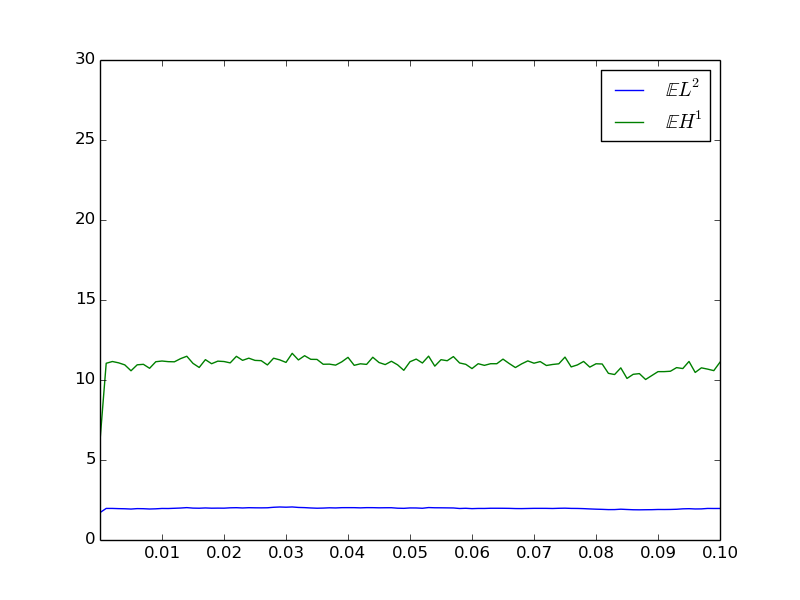}}}%
    \caption{$\mathbb{E} L^2$ and $\mathbb{E} H^1$ stability curves (one sample): $\epsilon = 0.1$, $h \approx 0.044$, and $\tau = 0.001$.}
    \label{fig:my_label1}
\end{figure}

\captionsetup[subfigure]{labelformat=empty}
\begin{figure}[h!]%
    \renewcommand{\thefigure}{5.\arabic{figure}}
    \pagecolor{white}
    \centering
    \subfloat[\centering $\delta = 1$]{{\includegraphics[width=6cm]{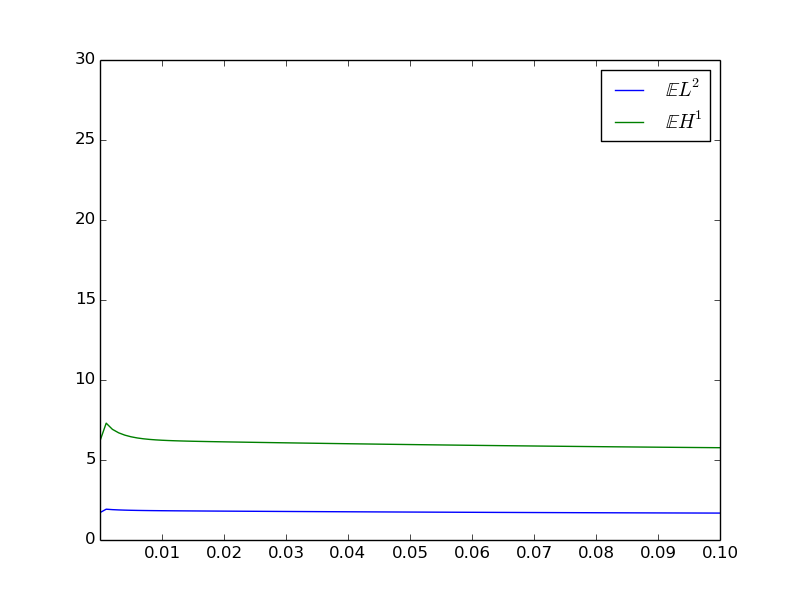}}}%
    \qquad
    \subfloat[\centering $\delta = 10$]{{\includegraphics[width=6cm]{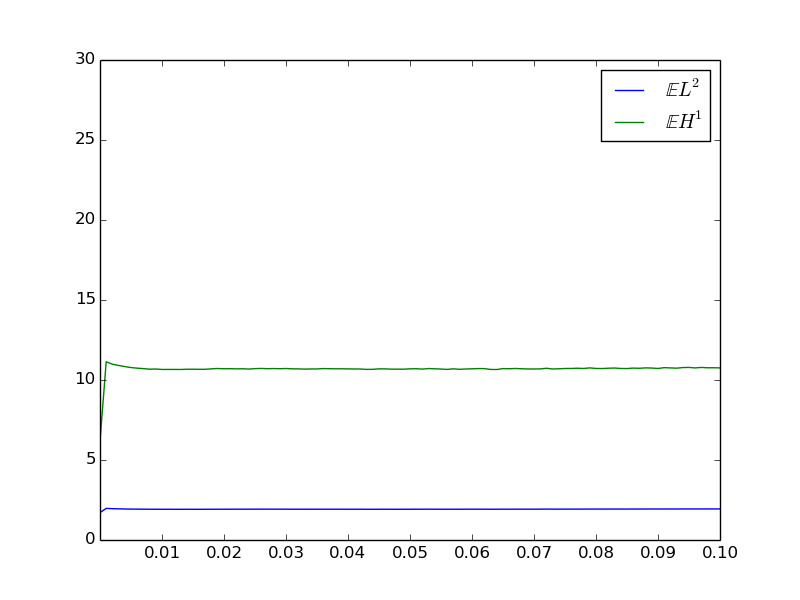}}}%
    \caption{$\mathbb{E} L^2$ and $\mathbb{E} H^1$ stability curves (average): $\epsilon = 0.1$, $h \approx 0.044$, and $\tau = 0.001$.}
    \label{fig:my_label2}
\end{figure}
\newpage
Then below these, the figures (\ref{fig:my_label3}) and (\ref{fig:my_label4}) show the zero-level sets of one sample and the average of samples, respectively. We can see that the evolution is of a shrinking circle that stabilizes at the final time $T = 0.1$.
\captionsetup[subfigure]{labelformat=empty}
\begin{figure}[h!]%
    \renewcommand{\thefigure}{5.\arabic{figure}}
    \pagecolor{white}
    \centering
    \subfloat[\centering $\delta = 1$]{{\includegraphics[width=6cm]{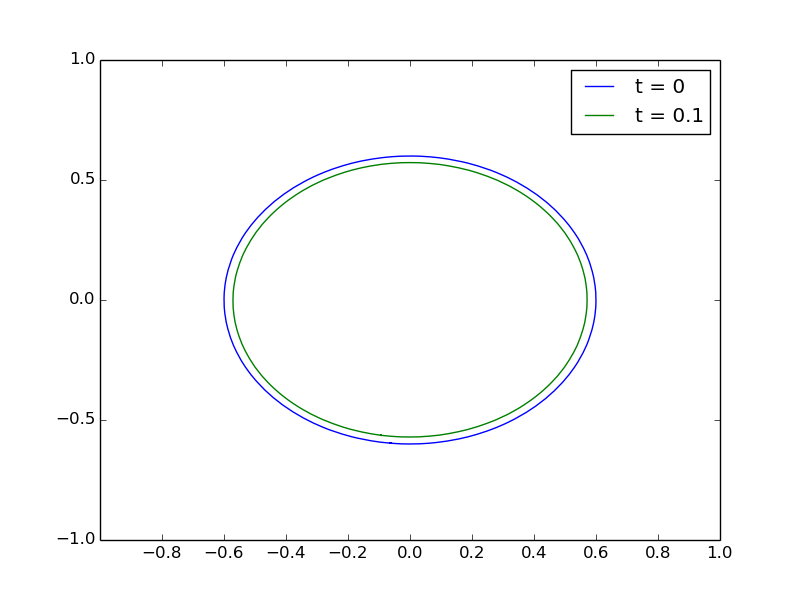}}}%
    \qquad
    \subfloat[\centering $\delta = 10$]{{\includegraphics[width=6cm]{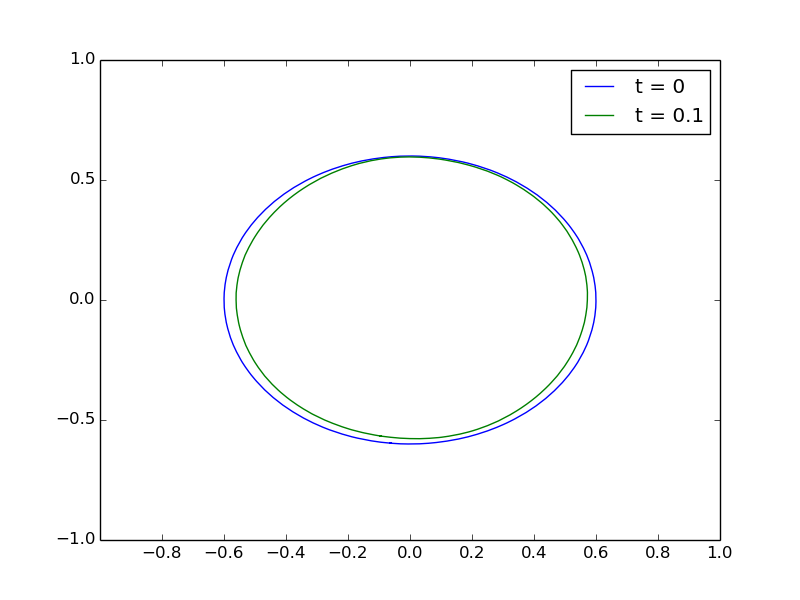}}}%
    \caption{Zero-level sets (one sample): $\epsilon = 0.1$, $h \approx 0.044$, and $\tau = 0.001$.}
    \label{fig:my_label3}
\end{figure}

\captionsetup[subfigure]{labelformat=empty}
\begin{figure}[h!]%
    \renewcommand{\thefigure}{5.\arabic{figure}}
    \pagecolor{white}
    \centering
    \subfloat[\centering $\delta = 1$]{{\includegraphics[width=6cm]{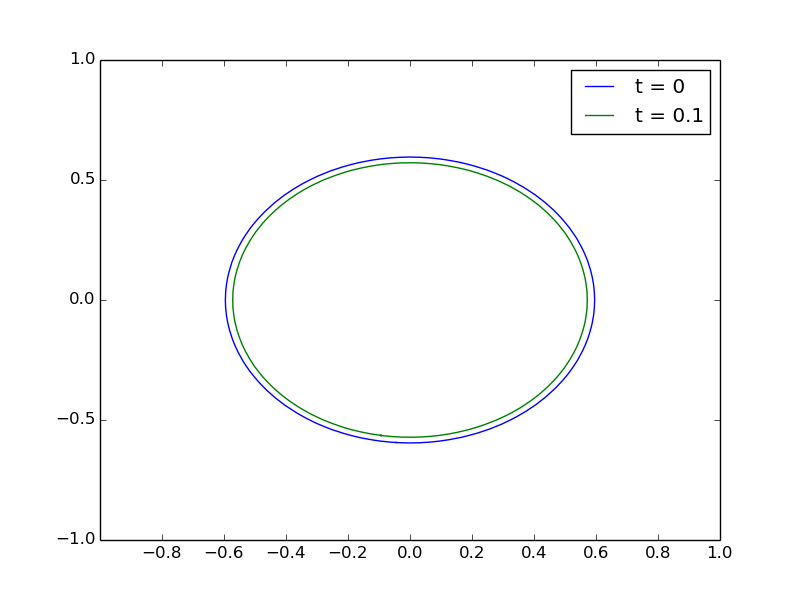}}}%
    \qquad
    \subfloat[\centering $\delta = 10$]{{\includegraphics[width=6cm]{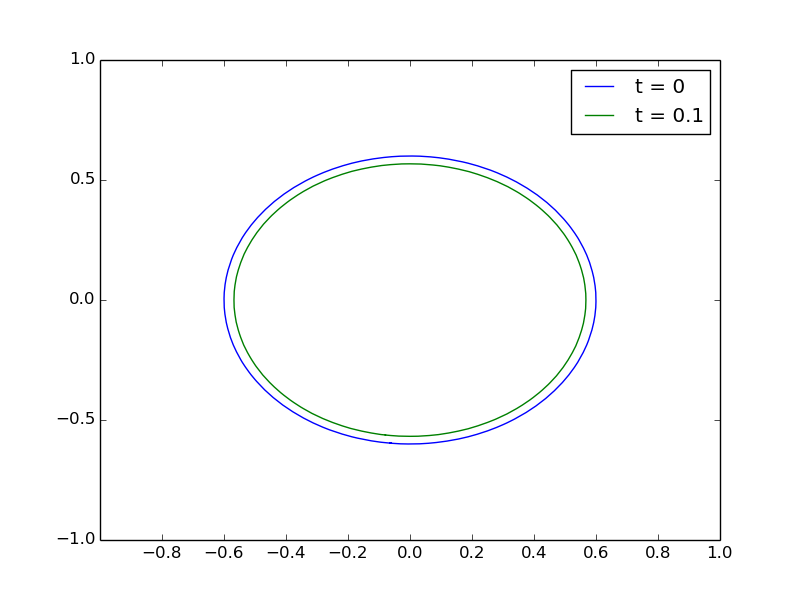}}}%
    \caption{Zero-level sets (average): $\epsilon = 0.1$, $h \approx 0.044$, and $\tau = 0.001$.}
    \label{fig:my_label4}
\end{figure}
\newpage
\textbf{Test 2:} For this test, we used the initial condition $u^0_h = P_hu_0$ where
$$
u_0(x, y) = \tanh \left ( \frac{\sqrt{(x/0.7)^2 + (y/0.65)^2} - 1}{\sqrt{2} \epsilon} \right )
$$
and again, the diffusion term is $g(u) = u$. The figures (\ref{fig:my_label5}) and (\ref{fig:my_label6}) show the $\mathbb{E} L^2$ and $\mathbb{E} H^1$ stability results of the numerical solution in one sample and the average of samples, respectively. Since the initial condition takes time to become a circle, the curves in these figures require a larger final time than in the previous test to stabilize.
\captionsetup[subfigure]{labelformat=empty}
\begin{figure}[h!]%
    \renewcommand{\thefigure}{5.\arabic{figure}}
    \pagecolor{white}
    \centering
    \subfloat[\centering $\delta = 1$]{{\includegraphics[width=6cm]{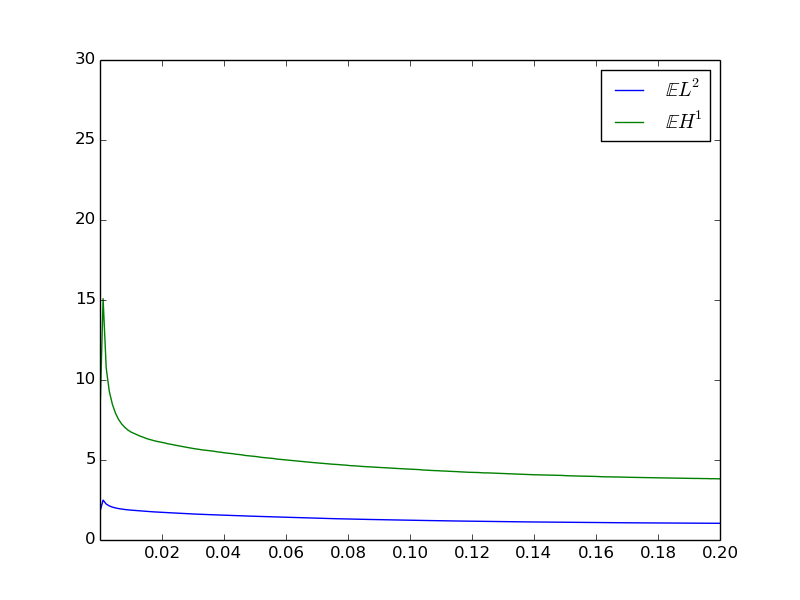}}}%
    \qquad
    \subfloat[\centering $\delta = 10$]{{\includegraphics[width=6cm]{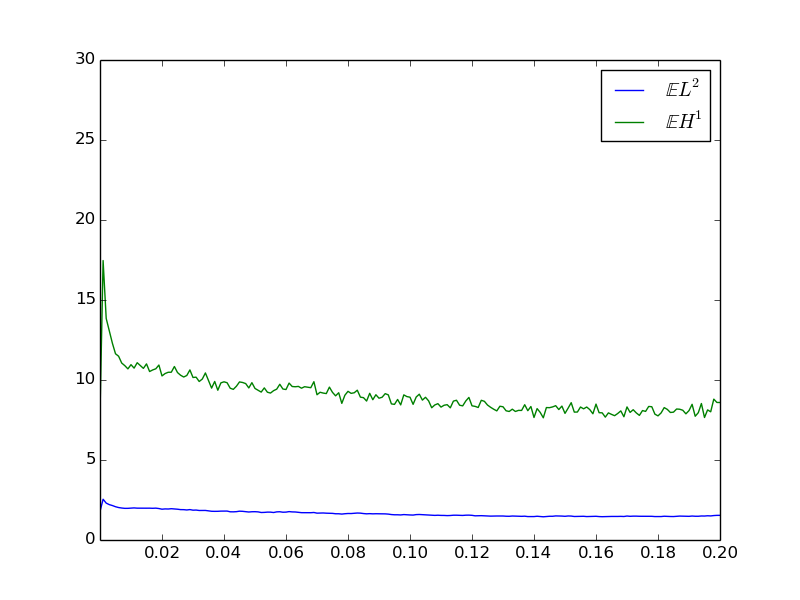}}}%
    \caption{$\mathbb{E} L^2$ and $\mathbb{E} H^1$ stability curves (one sample): $\epsilon = 0.1$, $h \approx 0.044$, and $\tau = 0.001$.}
    \label{fig:my_label5}
\end{figure}

\captionsetup[subfigure]{labelformat=empty}
\begin{figure}[h!]%
    \renewcommand{\thefigure}{5.\arabic{figure}}
    \pagecolor{white}
    \centering
    \subfloat[\centering $\delta = 1$]{{\includegraphics[width=6cm]{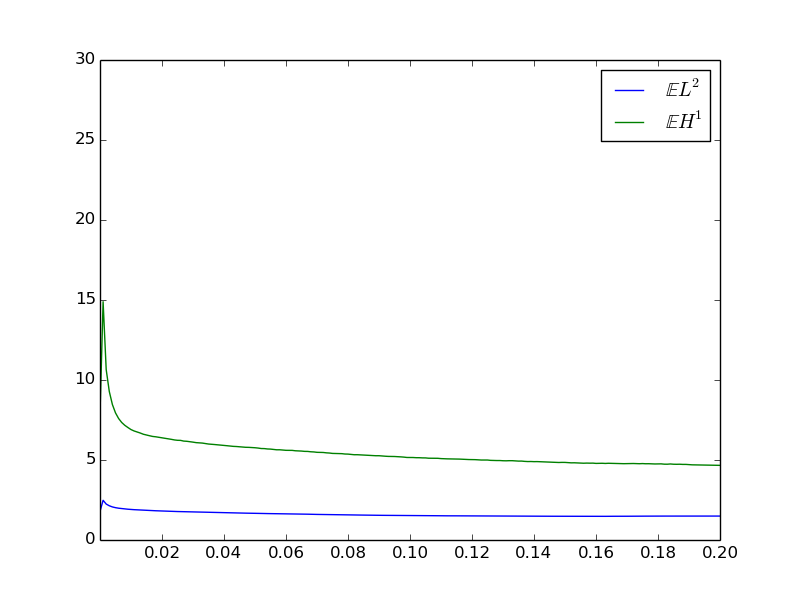}}}%
    \qquad
    \subfloat[\centering $\delta = 10$]{{\includegraphics[width=6cm]{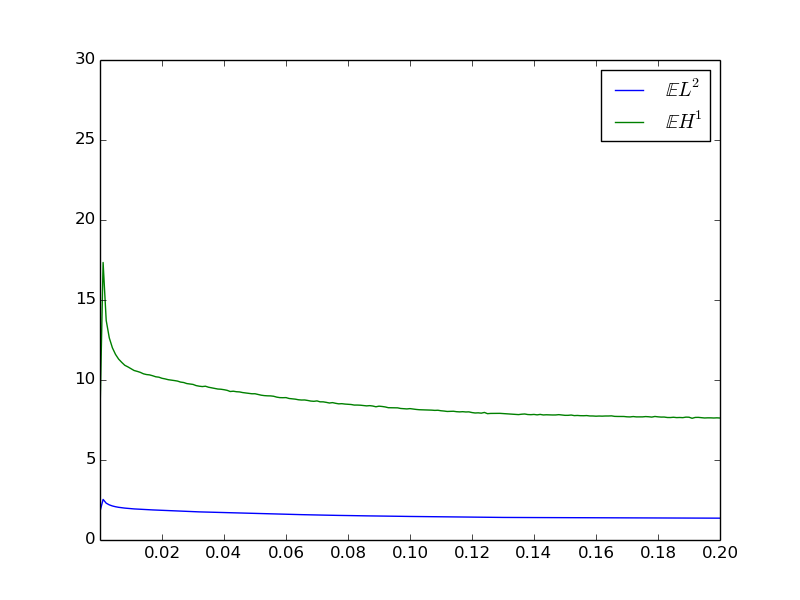}}}%
    \caption{$\mathbb{E} L^2$ and $\mathbb{E} H^1$ stability curves (average): $\epsilon = 0.1$, $h \approx 0.044$, and $\tau = 0.001$.}
    \label{fig:my_label6}
\end{figure}
Then below these, the figures (\ref{fig:my_label7}) and (\ref{fig:my_label8}) show zero-level sets of one sample and the average of samples, respectively. The numerical solutions approach a stable circle, which occurs faster for larger diffusion intensities.
\captionsetup[subfigure]{labelformat=empty}
\begin{figure}[h!]%
    \renewcommand{\thefigure}{5.\arabic{figure}}
    \pagecolor{white}
    \centering
    \subfloat[\centering $\delta = 1$]{{\includegraphics[width=6cm]{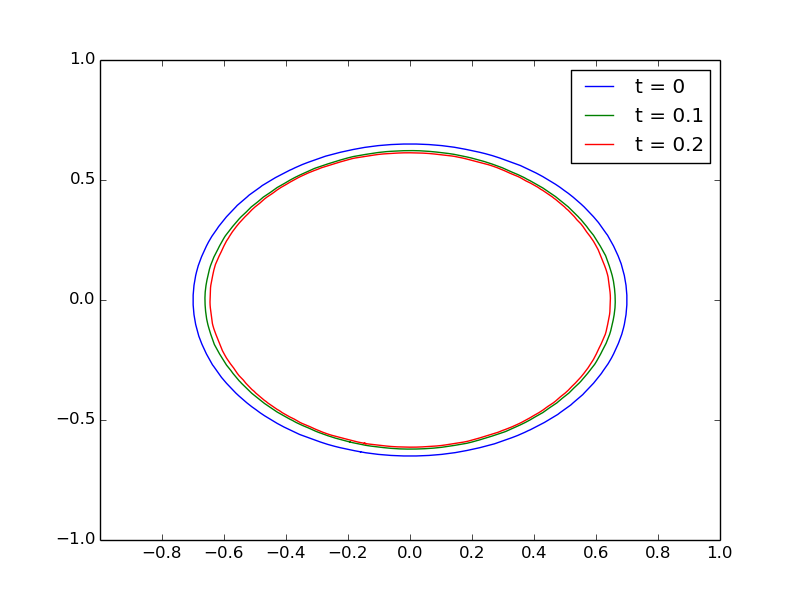}}}%
    \qquad
    \subfloat[\centering $\delta = 10$]{{\includegraphics[width=6cm]{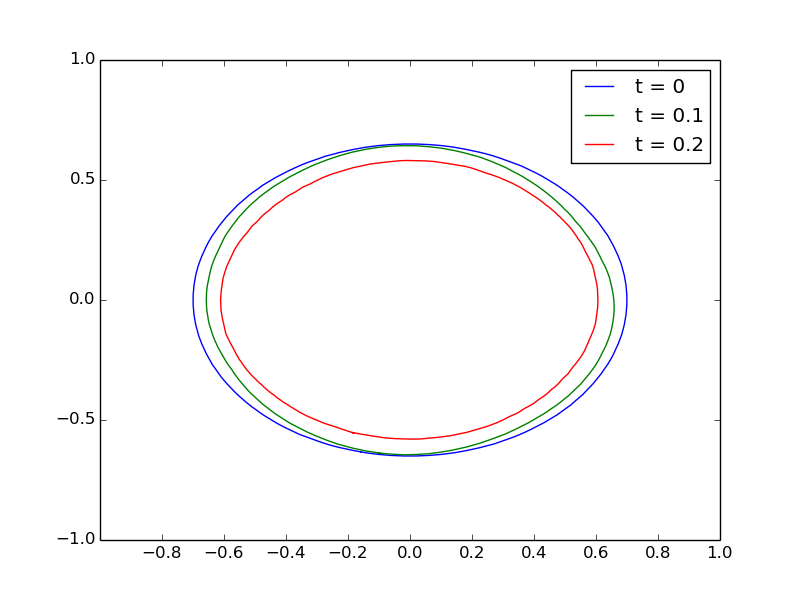}}}%
    \caption{Zero-level sets (one sample): $\epsilon = 0.1$, $h \approx 0.044$, and $\tau = 0.001$.}
    \label{fig:my_label7}
\end{figure}

\captionsetup[subfigure]{labelformat=empty}
\begin{figure}[h!]%
    \renewcommand{\thefigure}{5.\arabic{figure}}
    \pagecolor{white}
    \centering
    \subfloat[\centering $\delta = 1$]{{\includegraphics[width=6cm]{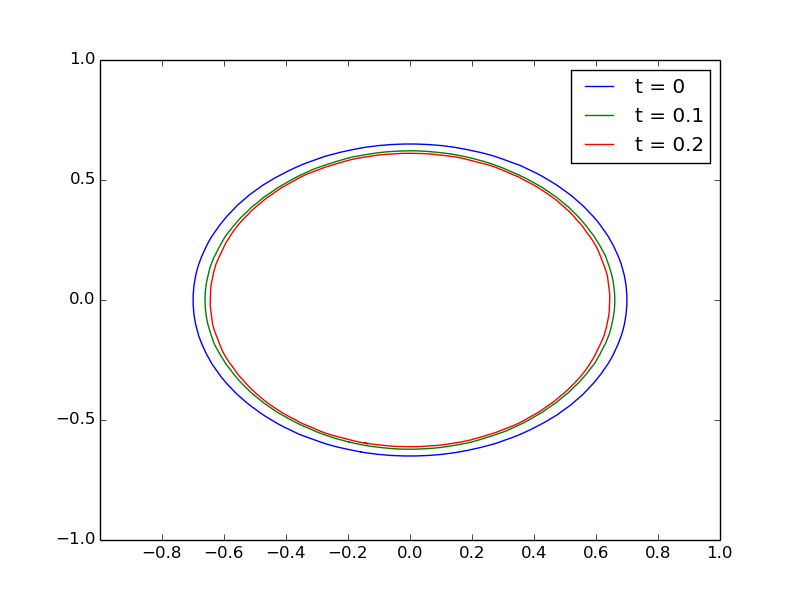}}}%
    \qquad
    \subfloat[\centering $\delta = 10$]{{\includegraphics[width=6cm]{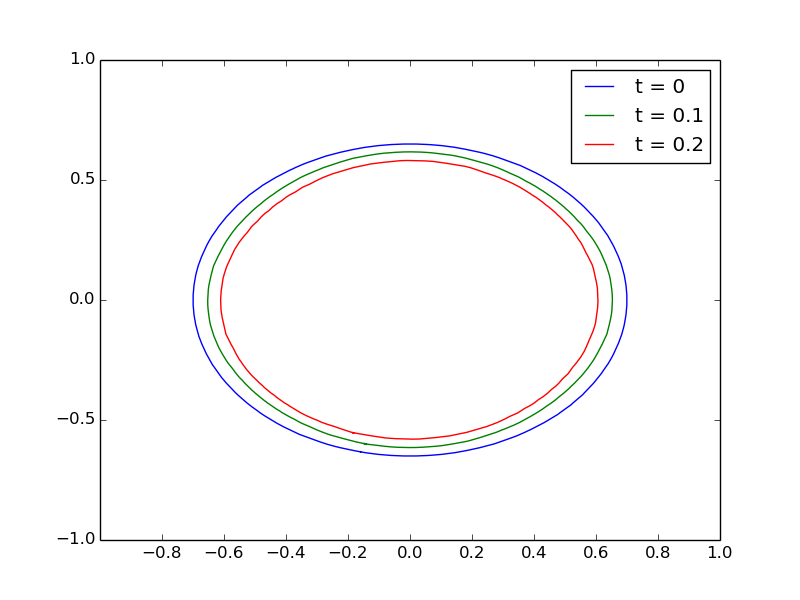}}}%
    \caption{Zero-level sets (average): $\epsilon = 0.1$, $h \approx 0.044$, and $\tau = 0.001$.}
    \label{fig:my_label8}
\end{figure}
\newpage
\textbf{Test 3:}
For this test, we used the initial condition $u^0_h = P_hu_0$ where
$$
u_0(x, y) = \tanh \left ( \frac{(\sqrt{x^2/0.7 + y^2/0.1} - 1)(\sqrt{x^2/0.1 + y^2/0.7} - 1)}{\sqrt{2}\epsilon} \right )
$$
with $\epsilon = 0.05$ and the diffusion term is $g(u) = \sqrt{u^2 + 1}$.
We computed the $L^{\infty}\mathbb{E} L^2 $,
and $\mathbb{E}L^2 H^1$ errors which denote
\begin{align}\label{eqn_error}
\left ( \max_{0 \leq n \leq N} \mathbb{E} \left[ \lVert E^n \rVert^2_{L^2(\mathcal{D})} \right] \right )^{1/2} \text{and} \quad \left (\mathbb{E} \left[ \tau\sum_{n=1}^N \lVert \nabla E^n \rVert^2_{L^2(\mathcal{D})}\right] \right)^{1/2},
\end{align}
respectively. Table (\ref{fig:Table_one}) below contains these errors. The final time is $T = 1/10^4$.
\vspace{1cm}
\begin{table}[h!]
\centering
\begin{tabular}{|l|l|l|l|l|l|l|}
\hline
$h$              & $L^{\infty} \mathbb{E} L^2$ & Order                      & $\mathbb{E}L^2H^1$ & Order                        \\ \hline
$0.2\sqrt{2}$    & 0.89614673     &                    & 13.81865669 &           \\ \hline
$0.1\sqrt{2}$    & 0.23059268     & 1.95838825         & 7.212195952 & 0.93810688         \\ \hline
$0.05\sqrt{2}$   & 0.06355413     & 1.85928886         & 3.674838827 & 0.97275762         \\ \hline
$0.025\sqrt{2}$  & 0.01715789     & 1.88911367         & 1.946229545 & 0.91699910  \\ \hline
\end{tabular}
\caption{Strong spatial errors and error orders: $\epsilon = 0.05$, $\delta = 1$, $\tau = 1/10^6$.}
\label{fig:Table_one}
\end{table}

% \begin{table}[H]
% \centering
% \begin{tabular}{|l|l|l|l|l|l|l|}
% \hline
% $h$              & $L^{\infty} \mathbb{E} L^2$ & Order                      & $\mathbb{E}L^2H^1$ & Order                        \\ \hline
% $0.2\sqrt{2}$    & 0.89615831     &                    & 13.81838842 &           \\ \hline
% $0.1\sqrt{2}$    & 0.23060130     & 1.95835296         & 7.21224573 & 0.93806891         \\ \hline
% $0.05\sqrt{2}$   & 0.06355038     & 1.85942799         & 3.67485053 & 0.97276298         \\ \hline
% $0.025\sqrt{2}$  & 0.017167881     & 1.88818878         & 1.94624048 & 0.91699559  \\ \hline
% \end{tabular}
% \caption{Weak spatial errors and error orders: $\epsilon = 0.05$, $\delta = 1$, $\tau = 1/10^6$}
% \label{fig:Table_two}
% \end{table}

%\newpage
\section{Conclusion}\label{sec14}
In this manuscript, we design a numerical scheme for the stochastic Cahn-Hilliard equation with multiplicative noise. The scheme utilizes the interpolation operator to handle the interaction between the drift term and the diffusion term, and is proven to maintain some stability results and higher moment results. Based on these stability results, we construct a probability one set such that the error estimates in the discrete $H^{-1}$-norm hold on this set.

Future work will remove the probability one set, i.e., establishing the error estimates in the whole probability space. The key step is to prove the higher moment bounds for $H^1$-norm. This is an open question for this type of stochastic Cahn-Hilliard equation and some new techniques need to be brought in.

%Conclusions may be used to restate your hypothesis or research question, restate your major findings, explain the relevance and the added value of your work, highlight any limitations of your study, describe future directions for research and recommendations. 
%
%In some disciplines use of Discussion or 'Conclusion' is interchangeable. It is not mandatory to use both. Please refer to Journal-level guidance for any specific requirements. 
%\\
%Supplementary information
%\\
%If your article has accompanying supplementary file/s please state so here. 
%
%Authors reporting data from electrophoretic gels and blots should supply the full unprocessed scans for key as part of their Supplementary information. This may be requested by the editorial team/s if it is missing.
%
%Please refer to Journal-level guidance for any specific requirements.
%\\
\section*{Acknowledgments}
The work of Y. Li and C. Prachniak was partially supported by the National Science Foundation under grant DMS-2110728. The work of Y. Zhang was partially supported by the National Science Foundation under grant DMS-2111004. We would like to thank the anonymous referees for their valuable comments and suggestions.

\end{document}